%% file: main_gtlw.tex
\documentclass[final,onefignum,onetabnum,a4paper]{article}

\input{main_shared}

\begin{document}

\maketitle

\begin{abstract}
\input{0_abstract}

\end{abstract}

\input{1_Introduction}  

\input{2_BDflux} 

\input{3_Disflux} 

\input{4_lwtheorem}

\input{5_NT}

\input{6_conclusion}

\section{Acknowledgements}
This work was partially supported by the German Science Foundation (DFG) under Grant SO 363/14-1 (Klein) and the Gutenberg Research College, JGU Mainz (\"Offner). 

\bibliographystyle{plain}
\bibliography{literature}

\input{FluxM}

\end{document}

%% file: main_shared.tex
\usepackage{lipsum}
\usepackage{amsfonts}
\usepackage{epstopdf}
\usepackage{hyperref}
\usepackage[a4paper, top=1in]{geometry}
\usepackage{datetime}
\newdateformat{monthyeardate}{%
  \monthname[\THEMONTH] \THEDAY, \THEYEAR}

\usepackage{xcolor}

\usepackage{amsmath}
\allowdisplaybreaks
\usepackage{amssymb}
\usepackage{commath}
\usepackage{mathtools}
\usepackage{bbm}

\usepackage{color}
\usepackage{graphicx}
\usepackage[small]{caption}
\usepackage{subcaption}

\usepackage{relsize}
\usepackage{adjustbox}
\usepackage{algorithm}
\usepackage[noend]{algpseudocode}
\usepackage{booktabs}
\usepackage{tikz}
\usetikzlibrary{decorations.pathreplacing,calligraphy}

\usepackage{verbatim}
\usepackage{ulem}

\usepackage{enumitem}
\setlist[enumerate]{leftmargin=.5in}
\setlist[itemize]{leftmargin=.5in}

\title{Entropy conservative high-order fluxes in the presence of boundaries}

\author{
Simon-Christian Klein\thanks{Institute of Mathematics, Technical University Brunswick, Brunswick, Germany, (\url{simon-christian.klein@tu-braunschweig.de})}
\and 
Philipp \"Offner\thanks{Institute of Mathematics, Johannes Gutenberg University, Mainz, Germany, (\url{poeffner@uni-mainz.de},)} 
}

\usepackage{amsthm}
\theoremstyle{definition}
\newtheorem{definition}{Definition}
\theoremstyle{theorem}
\newtheorem{lemma}{Lemma}
\theoremstyle{remark}

\theoremstyle{theorem}
\newtheorem{theorem}{Theorem}
\newtheorem{remark}{Remark}

\newcommand{\vd}{\mathrm{d}}
\newcommand{\intd}{\, \mathrm{d}}

\newcommand{\Z}{\mathbb{Z}}
\newcommand{\R}{\mathbb{R}} 
\renewcommand{\epsilon}{\varepsilon}

\newcommand{\fnum}{f^{\mathrm{num}}}
\newcommand{\ftp}{h}
\newcommand{\Ftp}{H}

\newcommand{\derive}[2] {\frac{\partial {#1} }{\partial {#2}}}
\newcommand{\derd}[2]{\frac{\vd {#1}}{\vd {#2}}}
\newcommand{\sset}[1]{ \{ #1\}}
\newcommand{\skp}[1]{\left \langle {#1} \right \rangle}
\newcommand{\zset}{{\sset{-z+1, \dots, 2p-z}}}
\newcommand{\iset}{{\zset \times \zset}}
\DeclareMathOperator{\Tr}{\mathrm{T_r}}
\DeclareMathOperator{\Tl}{\mathrm{T_l}}
\DeclareMathOperator{\Er}{\mathrm{E_r}}
\DeclareMathOperator{\El}{\mathrm{E_l}}

\DeclareMathOperator{\bigO}{\mathcal{O}}

%% file: 0_abstract.tex
In this paper, we propose a novel development in the context of entropy stable finite-volume/finite-difference schemes. 
In the first part, we focus on the construction of high-order entropy conservative fluxes. 
Already in \cite{lefloch2002fully}, the authors have generalized the second order accurate 
entropy conservative numerical fluxes proposed by Tadmor to  high-order ($2p$) by a simple centered linear combination. 
We generalize this result additionally to non-centered flux combinations which is in particular favorable if non-periodic boundary conditions are needed. 
In the second part, a Lax-Wendroff theorem for the combination of these fluxes and the entropy dissipation steering from \cite{klein2022using} is proven. 
In numerical simulations, we verify all of our theoretical findings.

%% file: 1_Introduction.tex
\section{Introduction}\label{se_int}

Hyperbolic conservation/balance laws play a fundamental role within mathematical models for various physical processes, 
including fluid mechanics, electromagnetism and wave phenomena.  However, especially nonlinear hyperbolic conservation laws cannot be solved 
by pen and paper. Therefore, suitable and efficient numerical methods have to be applied.
Finite-Volume (FV) methods are one of the most used approaches in this context due to their excellent robustness properties and are the foundation of many modern 
algorithms \cite{abgrall2016handbook}. 
They can be also extended to high-order by the usage of suitable numerical fluxes. Nowadays one is not only interested in robust schemes discretising a given conservation law, but rather desires methods that are compatible with additional underlying physical laws.
Especially, the construction of entropy conservative and dissipative schemes have attracted a lot of attention recently as can be seen in the massive literature focusing on this topic, cf. \cite{abgrall2021relaxation, abgrall2022reinterpretation,   clain2011high, fisher2013high_2,  fisher2013discretely, kuzmin2022bound, fjordholm2012arbi, fjordholm2012entropy,  jameson2008construction, ray2016entropy, sonar1992entropy, klein2022using, klein2022stabilizing} and references therein.
To construct entropy stable schemes many techniques are based on a general description by Tadmor \cite{tadmor1987numerical,tadmor2003entropy},  where he presented how to construct entropy conservative or dissipative schemes by using special two-point fluxes. Tadmor's schemes are second order accurate \cite{tadmor1987numerical,tadmor2003entropy}. By using simple centered linear combinations of entropy conservative two-points fluxes, LeFloch, Mercier and Rhode generalized Tadmor's approach to $2p$ order \cite{lefloch2002fully} and  many entropy conservative/dissipative schemes use these fluxes \cite{fjordholm2012entropy,  lefloch2002fully, ranocha2018generalised, klein2022using}, focusing especially on periodic problems. In the following paper, we generalize the results from \cite{lefloch2002fully} by using a non-centered flux combination that is in particular favorable if non-periodic boundary conditions are needed. \\

Therefore, the paper is organized as follows: In Section \ref{se:notation} we introduce the notation and setting used in the manuscript. Then, in Section \ref{se:bdflux}, we repeat first the result from \cite{lefloch2002fully} about high-order entropy conservative 
fluxes for periodic domains. Next, we extend the result by using a non-centered flux combination, prove the existence of an entropy conserving entropy flux for such combinations and explain their usage if solutions on non-periodic domains are calculated. We extend this investigation to entropy dissipative fluxes. As a last theoretical result, in Section \ref{se:lwthm} we demonstrate that a Lax-Wendroff type theorem holds for our FV schemes if we apply our new flux combinations. 
In numerical simulations in section \ref{se_NT}, we verify all of our theoretical findings and a summary closes this publication. 

%% file: 2_BDflux.tex
\section{Notation}\label{se:notation}

We are interested in the numerical approximation of solutions of general systems of conservation laws of the form\footnote{For simplicity, we focus in the first part on one-dimensional problems.}
\begin{equation}\label{eq_con}
\partial_t u +\partial_x f(u)=0, \quad u=u(x,t) \subset \R^N, \quad x \in \Omega \subset \R,\; t>0.
\end{equation}
equipped with adequate initial and boundary conditions. 
The conservation law \eqref{eq_con} shall be endowed with a smooth entropy-entropy flux pair $(U,F): \R^N \to \R^2$ 
which satisfies for a smooth solution $u$ the additional conservation law 
\begin{equation}\label{eq_entr}
\partial_t U(u) +\partial_x F(u)\stackrel{(\leq)}{=}0,
\end{equation}
whereas for weak solutions $u$, we obtain the inequality in \eqref{eq_entr} \cite{Lax71}. 
Further, we denote by $v(u)= \nabla U(u)$ the entropy variable associated with the given entropy $U$. 
If $U$ is strictly convex, the mapping between $u$ and $v$ is one-to-one. Therefore, we can work directly with $v$. 
A remapping is possible as described in \cite{Harten83b}, and a scalar-valued function, called potential, $\psi$ exists with
$
\psi= v \cdot f-F.
$ In the following, we will always assume a strictly convex entropy function additionally. \\
Translating this to the discrete framework, numerical schemes should not only be constructed to satisfy \eqref{eq_con} but rather 
ensure as well \eqref{eq_entr}. In his pioneer work \cite{tadmor1987numerical}, Tadmor constructed semi-discrete finite difference schemes
satisfying a discrete form of \eqref{eq_entr}. We follow the classical notation and speak about
 entropy conservative schemes if the equality in \eqref{eq_entr} is fulfilled whereas a scheme is called entropy dissipative in case of an
  inequality. Up to now,  
  a lot of attention has been allocated to the construction of entropy conservative/stable schemes in different frameworks
  \cite{abgrall2018general, abgrall2021analysis_II, chan2018discretely, fjordholm2012arbi, gassner2016split, hesthaven2019entropy}.
  Here, we follow the classical finite volume (FV) (or finite difference (FD)) setting used in 
  \cite{tadmor1987numerical, lefloch2002fully, fjordholm2012arbi}.
  Using uniform cells with cell center $x_j$ and mesh parameter $\Delta x=x_{j+\frac 1 2}-x_{j - \frac 1 2}$. 
  Our FV/FD schemes are given in semidisrete form 
  via 
  \begin{equation} \label{eq_semidiscrete}
\derd{u_k}{t}= \frac{1}{\Delta x} \left(f_{k-\frac 1 2}-f_{k+\frac 1 2} \right),
  \end{equation}
where $f_{k \pm \frac 1 2}$ denotes a numerical flux and we can use up to $2p+1$ cell states as arguments, i.e. a centered flux with $2 p$ arguments is given as
$
		f_{k+\frac 1 2} := f(u_{k-p+1}, \dots,  u_{k+p}).
$
Such a numerical flux is called  consistent with the continuous flux  if $ f(u, \dots,  u) =f(u)$ holds.
   Analogously,  a consistent entropy numerical fluxes  is given through
  $F_{k+\frac 1 2} = F(u_{k-p+1}, \dots,  u_{k+p}) $. 
  The key ingredient of our FV/FD schemes are the selected numerical fluxes.  
  Working with two-point fluxes, i.g. $f_{k+\frac 1 2} = \ftp(u_{k}, u_{k+1})$ and following 
    \cite{tadmor1987numerical}, we a  numerical flux is
 entropy conservative (dissipative)  if 
     \begin{equation}\label{entropy }
    \left(v_{k+1}-v_k\right) h(u_{k}, u_{k+1}) -\left(\Psi_{k+1}-\Psi_k \right)  \stackrel{(\leq )}{=}0
    \end{equation}
is satisfied. Using entropy conservative (dissipative) fluxes in our FV/FD discretization \eqref{eq_semidiscrete}, it is easy to show \cite{tadmor1987numerical} that 
 the resulting  methods are entropy conservative (dissipative) with numerical entropy flux 
 \[
 \Ftp(u_k,u_{k+1})= \frac 1 2  \left(v_{k+1}+v_k\right) \ftp(u_{k}, u_{k+1})-  \frac 1 2 \left(\Psi_{k+1}+\Psi_k \right).
 \]
We will from now on denote entropy conservative two-point fluxes as $\ftp(u_k, u_{k+1})$, whereas entropy dissipative two-point fluxes, like the Godunov flux, are written as
$g(u_k, u_{k+1})$. 

\section{Entropy conservative boundary fluxes} \label{se:bdflux}

The aforementioned entropy conservative fluxes are second order accurate.
The authors of \cite{lefloch2002fully} demonstrated how to construct a $2p$ order accurate entropy conservative flux by using a linear combination of 
Tadmor's two-point fluxes and to ensure the entropy conservation for their considered scheme for periodic domains. 
They further constructed fully discrete schemes of second and third order by applying implicit RK methods and validate the Lax-Wendroff theorem. However, we concentrate on the semi-discrete version only and generalize the used combination to be able to handle also non-periodic problems.
Further, we also validate a Lax-Wendroff type theorem. For the fully discrete setting, one can work with special implicit methods or the relaxation approach from \cite{ranocha2020relaxation}. However, this is not a topic of the current investigation.
We start by a result from \cite{fisher2013high_2} which was later generalized in \cite{ranocha2018comparison}.

\begin{lemma}[Symmetric flux expansions {\cite{fisher2013high_2, ranocha2018comparison}}] \label{thm:rfe}

If the numerical flux $\ftp$ is smooth, consistent with the flux $f$, and symmetric, a power series expansion of $\ftp$ can be written as
	\begin{equation}
		\ftp(u(x_0), u(x)) = f(u(x_0)) + \frac {f^{\prime}(u(x_0))} {2} (x - x_0) + \sum_{k\geq 2} \alpha_k (x-x_0)^k,
	\end{equation}
where $\alpha_k$ are scalar coefficients and the multi-index notation is used. 
\end{lemma}
\begin{remark}
We would like to point out that this result was used in a  different context. 
Fisher and Carpenter generalized the needed tools for high-order entropy conservative fluxes in the context of summation-by-parts (SBP) 
operators and applied a flux-differencing (telescoping) ansatz where they changed the volume fluxes 
inside each element (or block) depending on the used FD (or DG) framework. 
This approach has attract many  attentions since then, cf.\cite{fisher2013discretely, gassner2016split, carpenter2016entropy, chen2017entropy, crean2018entropy}.
\end{remark}
We use his result to simplify the proof from  \cite{lefloch2002fully} and the theorem is formulated as follows:
\begin{theorem}[High-order entropy conservative fluxes for periodic domains {\cite{lefloch2002fully}}]
	\label{thm:LMR}
	Let $\ftp(u_i,u_j)$ be a smooth, consistent, symmetric and entropy conservative numerical two-point flux. Then, the numerical flux 
	 defined by 
	\begin{equation} \label{eq:lmrflux}
		f_{k+\frac 1 2} = f(u_{k-p+1}, \dots,  u_{k+p}) = \sum_{r= 1}^p c^r_p \sum_{q=0}^{r-1} \ftp(u_{k-q}, u_{k+r-q})
	\end{equation}
	is a semi-discrete entropy conservative flux of order $2p$ with entropy flux
	\begin{equation}\label{eq_equation_flux}
		F_{k+\frac 1 2} = F(u_{k-p+1}, \dots,  u_{k+p}) = \sum_{r= 1}^p c^r_p \sum_{q=0}^{r-1} \Ftp(u_{k-q}, u_{k+r-q}),
	\end{equation}
	if the coefficients satisfy
	\[
		\sum_{r=1}^p r c_p^r = 1 ,\quad \sum_{r=1}^p c_p^r r^{2k+1} = 0, \quad \forall k \in \sset{1, \dots, p-1}.
		\]
\end{theorem}
\begin{proof}
	The difference of two fluxes can be written as
	
	\[
	 \derd{u_k}{t} = \frac{f_{k-1/2} - f_{k+1/2} }{\Delta x} =\sum_{r = 1}^p c_p^r  \frac{\ftp(u_{k-r}, u_{k}) - \ftp(u_{k}, u_{k + r})}{\Delta x},
	\]
	as all other contributions cancel out.
	First, we present the entropy conservation.  Using elementary calculations, we obtain 
	\[
	\begin{aligned}
		\derd{U(u_k)}{t} = &\skp{\derd{U}{u},\derd{u_k}{t}}
		=  \skp{v_k,  \frac{1}{\Delta x} \left(f_{k-\frac 1 2}- f_{k+\frac 1 2} \right)}
		= \skp{v_k,  \sum_{r=1}^p c^r_p \frac{\ftp(u_{k-r}, u_{k}) - \ftp(u_{k}, u_{k+r})}{\Delta x} } \\
		= &
		\sum_{r=1}^p c^r_p \skp{ v_k, \frac{\ftp(u_{k-r}, u_{k}) - \ftp(u_{k}, u_{k+r})}{\Delta x} } 
		=
		\sum_{r=1}^p c^r_p  \frac{\Ftp(u_{k-r}, u_{k}) - \Ftp(u_{k}, u_{k+r})}{\Delta x}\\
		= & - \frac{1}{\Delta x} \left(F_{k+\frac 1 2}- F_{k-\frac 1 2} \right).
	\end{aligned}	
\]
	This demonstrates the semi-discrete per cell entropy conservation of the flux. \\
	To demonstrate the $2p$ order of accuracy, Lemma \ref{thm:rfe} is applied.  Note that $\alpha_0 = f(u(x_0))$ and $\alpha_1 = f'(u(x_0))$ hold, and therefore we can expand the fluxes on the grid  into
	\[
		\ftp(u_k, u_{k+r}) = f(u(x_k)) + \frac 1 2 f'(u(x_k)) \Delta x r + \sum_{k \geq 2} \alpha_k r^k (\Delta x)^k.
	\]
	By applying this series expansion, one finds 
	\[
	\begin{aligned}
		\derd{u_k}{t} 
		=&-\sum_{r=1}^p c_p^r\frac{ \sum_{l=0}^\infty \alpha_l (r^l - (-r)^l) (\Delta x)^l}{\Delta x}
		= -\sum_{r = 1}^p c_p^r\frac{ (f'(u_k) + f'(u_k)) r \Delta x} {2\Delta x}  \\
		&- \sum_{r = 1}^p c_p^r
		\frac{2 \sum_{l \geq 1}r^{2l+1}(\Delta x)^{2l+1}\alpha_{2l+1} }{\Delta x} 
		=- f'(u_k) \sum_{r=1}^p r c_p^r 
		-
		2 \sum_{l \geq 1} (\Delta x)^{2l} \sum_{r=1}^p c_p^r r^{2l+1}\alpha_{2l+1} \\
		=&-f'(u_k) + \bigO\left(\abs{\Delta x}^{2p}\right),
	\end{aligned}
	\]
	where we use that even components of the sum cancel out due to the symmetry of the numerical fluxes. Uneven components of the sum vanish because the condition $\sum_{r=1}^p c_p^r r^{2l+1}=0$ for $l \in \{1, \dots, p-1\}$ holds.

\end{proof}

One of the key ingredients in the argument above is that the linear combination of the numerical fluxes is centered around $x_k$. 
Due to this fact, they can only be used if enough cells around cell $k$ are available, for example if periodic boundary conditions are considered. 
One example of methods using those linear combination are the TeCNO schemes  developed in \cite{fjordholm2012entropy} with order above two. 
However, up to our knowledge for other boundary conditions, like inflow conditions, one falls back to the application of second order two-point fluxes .Therefore, we give now the following extension: 

\begin{definition}[Linear combined fluxes] \label{def:lcflux}
	Let $\ftp(u_L,u_R)$ be a two-point flux and \newline $A  \in \R^{\iset}$ be a matrix that satisfies 
	\[
	\forall  l \in Z = \{-z+1, \dots, 2p-z \}: A_{ll} = 0.
	\]
	We  define a new numerical flux through
	\begin{equation}\label{eq_combination}
	f(u_{k-z+1}, \dots, u_{k+2p-z}) := \sum_{l, m = -z+1}^{2p-z} A_{lm} \ftp(u_{k+l}, u_{k+m}),	
	\end{equation}
	and call it a linear combination numerical flux with stencil $k + Z = \{k-z+1, \dots, k+2p-z\}$.  Clearly, the value $z$ shifts the stencil of the designed  flux, and $z = p$ corresponds to a symmetric stencil. The set $Z$ contains the indices of the matrix $A$. These indices are also the offsets of the arguments of the two-point flux from the grid point $k$.

\end{definition}
Obviously, the fluxes \eqref{eq_combination} are a generalization of the fluxes constructed in \cite{lefloch2002fully} and used in Theorem \ref{thm:LMR}
\begin{remark}
 The requirement of a diagonal zero is not obviously needed. But we want to avoid that the two-point flux, evaluated with equal arguments, is used to approximate the analytical flux. Additionally, we observed in numerical simulations that this approach leads to an increase in the usable CFL number of the constructed schemes. 
 
\end{remark}
\begin{lemma}\label{lem:matrixflux}
	The high-order flux \eqref{eq:lmrflux} can be written as
	\begin{equation}
		f_{k+ \frac 1 2} = \sum_{l, m = -z+1}^{2p-z} A_{lm} \ftp(u_{k+l}, u_{k+m})
	\end{equation}
	using the matrix $A \in \R^{\iset}$ with
	$
	A_{l, m} = \frac{c_p^{\abs{l - m}}} 2
	$
	and $z = p$. 
\end{lemma}
Writing down the same flux for the variable $k = i +1$ amounts to shifting the matrix $A$. If one denotes with $Z = \zset$ the index set for the matrix $A$ and with $\Tr$ the index shift operator \[\Tr: \R^{Z \times Z} \to \R^{(Z-1) \times (Z-1)},\quad A \mapsto \tilde A, \quad \tilde A_{lm} = A_{l+1, m+1}\] 
and $\tilde z = z+1$ one finds
\[
\begin{aligned}
	f_{i + \frac 1 2}  = &\sum_{l, m = -z+1}^{2p-z} A_{lm} \ftp(u_{i+l}, u_{i+m}) &=& \sum_{l, m= -z + 1}^{2p-z} A_{l, m} \ftp(u_{k + l -1}, u_{k+m-1}) \\
	=& \sum_{l, m = - \tilde z + 1}^{2p-\tilde z} A_{l+1, m+1} \ftp(u_{k+l}, u_{k+m}) &=& \sum_{l, m = - \tilde z + 1}^{2p-\tilde z} (\Tr A)_{l, m} \ftp(u_{k+l}, u_{k+m}) = f_{k-\frac 1 2}.
\end{aligned}
\]
Please be aware of the fact that $l \to l-1$ and $m \to m-1$ were shifted in the above derivation between the first and second row. This also implies that the matrix $\Tr A$ has an index range lowered by one. The same applies to an operator $\Tl = \Tr^{-1}$ that is used to transfer the variable to the left. We will assume in the following theorem that $A_{kl} = A_{lk}$ holds, i.e. the coefficient matrices are symmetric. This is not a strong restriction as our used entropy conservative two-point fluxes are anyway symmetric. We will now regularly write sums as
\[
f_{i + \frac 1 2} = \sum_{l, m = -z}^{2p-z} (\Er A)_{lm} \ftp(u_{i+l}, u_{i+m}), \quad f_{i + \frac 1 2} = \sum_{l, m = -z+1}^{2p-z+1} (\El A)_{lm} \ftp(u_{i+l}, u_{i+m}).
\]
Note that the used summations start one index lower than before or ends one index higher than before. 
This will be needed to manipulate the difference of two fluxes one cell apart. We therefore use the embedding operator
\[
\Er: \R^{Z\times Z} \to \R^{\tilde Z \times \tilde Z}, A \mapsto \Er A, (\Er A)_{lm} = \begin{cases}A_{lm}& (l, m) \in Z^2 \\ 0& (l, m) \not \in Z^2 \end{cases}
\]
with $\tilde Z = Z \cup \{\max Z + 1 \}$ and the counterpart $\El$.
We will now, armed with the previous notation and lemmas, begin to derive boundary aware fluxes. Our first lemma gives sufficient conditions for two fluxes to provide a high-order scheme, when used together in a discretization.
\begin{lemma}[Non-centered high-order fluxes]\label{lem:HOBF}
Let $A, B \in \R^{Z\times Z}$. The fluxes defined by
	\begin{equation}
		f_{k - \frac 1 2} = \sum_{l,m =-z+1}^{2p-z} \Er\Tr A_{lm}\ftp(u_{k+l}, u_{k+m}),
		\label{eq:ansatzA}
	\end{equation}
	\begin{equation}
		f_{k + \frac 1 2} = \sum_{l,m =-z+1}^{2p-z} B_{lm}\ftp(u_{k+l}, u_{k+m})
		\label{eq:ansatzB}
	\end{equation}
	are linear combination fluxes of finite difference order $2p-1$ for $f$ with the same stencil $\sset{k - z, \dots, k+2p-z} = k + Z(z)$ if the coefficients satisfy the relations
	\begin{equation}
		\Er \Tr A_{lm} - B_{lm} = 0 , \quad \forall (l, m) \in (Z \setminus\sset{0}) \times (Z \setminus\sset{0}),
		\label{eq:taylorrestrict}
	\end{equation}
	\begin{equation}
		\begin{aligned}
			v_m = (\Er \Tr A_{0m} - B_{0m}) + (\Er \Tr A_{m0} - B_{m0}) \in \R^{Z}, \\
			\sum_{m = -z}^{2p-z} v_m {(m)^j} = 2\delta_{j1}, \quad j \in \sset{0, \dots, 2p-1}
			\label{eq:forderres},
		\end{aligned}
	\end{equation}
	\begin{equation}
		\forall l \in Z: A_{ll}= 0 \wedge v_0 = 0,
		\label{eq:lcfres}
	\end{equation}
	\begin{equation}
		\sum_{l,m = -z+1}^{2p-z} \Er \Tr A_{lm} = 1,
		\label{eq:fconres}
	\end{equation}
	and the flux $h(u_L u_R)$ is a consistent, smooth and symmetric numerical two-point flux.
	\begin{proof}
		We start with the ansatzes given in equation \eqref{eq:ansatzA} and \eqref{eq:ansatzB} above and look at their difference
		\begin{equation}
			f_{k + \frac 1 2} - f_{k - \frac 1 2} = \sum_{l,m =-z}^{2p-z} (B_{lm} - \Er \Tr A_{lm}) \ftp(u_{k+l}, u_{k+m}) = \sum_{l =-z}^{2p-z} (v_l)\ftp(u_{k}, u_{k+l}).
		\end{equation}
		We will use the series expansion of $f(u_k, u(x))$ from Lemma \ref{thm:rfe}, for which we would like $f(u_l, u_r)$ to satisfy either $u_l = u_k$ or $u_r = u_k$, as $k\Delta x $ will be the expansion point. The first set of restrictions on the coefficients, as given in equation \eqref{eq:taylorrestrict}, enables exactly these Taylor expansions. Therefore,
		\[
		(\Er \Tr A_{0l} + \Er \Tr A_{l0}) - (B_{0l} + B_{l0}) = v_l \iff \Er \Tr A_{0_l} - B_{0l} = \frac{v_l}{2}
		\]
		 are the only non-zero components in the differences between both matrices using the symmetry.
		The expansion is in fact given as
		\[
		\sum_{l =-z}^{2p-z} v_l \ftp(u_{k}, u_{k+l}) = \sum_{l =-z}^{2p-z} v_l\left(\sum_{j = 0}^\infty \alpha_j(u_k) l^j (\Delta x)^j\right) = \sum_{j = 0}^\infty	\sum_{l =-z}^{2p-z} v_l \alpha_j l^j (\Delta x)^j
		\]
		with $\alpha_0(u_k) = f(u_k)$, $\alpha_1(u_k) = \derive{f}{u}(u_k)/2$ by Theorem \ref{thm:rfe}.
		This should in turn equal $\Delta x\derive f x + \bigO((\Delta x)^{2p})$. Swapping the sums and comparing the coefficients of the derivatives leads to
		\begin{equation}
			\sum_{l=-z}^{2p-z} v_l  l^{j} = 2 \delta_{1j}, \quad j = 0, ..., 2p-1.
		\end{equation}
		
		This ensures the order of exactness of the flux. We further note that this equation for $j = 0$ implies 
		\begin{equation}
			\sum_{l,m = -z}^{2p-z} B_{lm} =\sum_{l,m = -z}^{2p-z}\Er \Tr A_{lm} + \sum_{l,m = -z}^{2p-z} B_{lm} - \Er \Tr A_{lm} = 1 + \frac 1 2 \sum_{l= -z}^{2p-z} v_{l} = 1.
		\end{equation}
		Therefore, both fluxes are conservative/consistent if $h$ is and  equation \eqref{eq:fconres} is satisfied. Finally, due to \eqref{eq:lcfres} and \eqref{eq:taylorrestrict}, both fluxes are clearly linear combined two-point fluxes in the sense of Definition \ref{def:lcflux}.
	\end{proof}
\end{lemma}
The last lemma  states sufficient conditions for two fluxes to be of high-order if used together and is similar to the one from  \cite{lefloch2002fully}. 
Finally, we demonstrate that our new combined fluxes can be as well entropy conservative. 
\begin{theorem}[Entropy conservative linear combination fluxes] \label{lem:ECF}
	The fluxes defined by
	\begin{equation}
		f_{k - \frac 1 2} = \sum_{l,m =-z+1}^{2p-z} \Er \Tr A_{lm}\ftp(u_{k+l}, u_{k+m}),
		\label{eq:ECansatzA}
	\end{equation}
	\begin{equation}
		f_{k + \frac 1 2} = \sum_{l,m =-z+1}^{2p-z} B_{lm}\ftp(u_{k+l}, u_{k+m})
		\label{eq:ECansatzB}
	\end{equation}
	are entropy conservative numerical fluxes with associated entropy fluxes
	\begin{equation}
		F_{k - \frac 1 2} = \sum_{l,m =-z+1}^{2p-z} \Er \Tr A_{lm}\Ftp(u_{k+l}, u_{k+m})
		\label{eq:EFansatzA},
	\end{equation}
	\begin{equation}
		F_{k + \frac 1 2} = \sum_{l,m =-z+1}^{2p-z} B_{lm}\Ftp(u_{k+l}, u_{k+m})
		\label{eq:EFansatzB}
	\end{equation}
	and the same stencil $\sset{k - z, \dots, k+2p-z} = k + Z(z)$ if the coefficients fulfil the relations
	\begin{equation}
		\sum_{l,m = -z}^{2p-z} TA_{lm} = 1,
		\label{eq:Efconres}
	\end{equation}
	\begin{equation}
		\Er \Tr A_{lm} - B_{lm} = 0 , \quad \forall (l, m) \in (Z \setminus\sset{0}) \times (Z \setminus\sset{0}) = N,
		\label{eq:Etaylorrestrict}
	\end{equation}
	\begin{equation}
		\forall l \in Z: A_{ll}= 0 \wedge v_0 = 0.
		\label{eq:Elcfres}
	\end{equation}
	\begin{proof}
		Our proof follows as a generalization of the proof from  \cite{lefloch2002fully}. Some more care has to be taken  as the used fluxes are \textit{not position independent}. 
		The needed compatibility relations are already a subset of the conditions from the previous Lemma \ref{lem:HOBF}. \\ 
		First, one can split the fluxes into components that are build out of linear combinations involving $u_k$ in the flux arguments, and the ones that do not. Therefore, we get 		\begin{equation}
		\begin{aligned}
			f_{k - \frac 1 2} =& \sum_{l,m=-z}^{2p - z} \Er \Tr A_{lm} \ftp(u_{k+l}, u_{k+m}) \\
			=& \sum_{l = -z}^{2p-z} \Er \Tr A_{l0} \ftp(u_{k+l}, u_{k}) &+& \sum_{l = -z}^{2p-z} \Er \Tr A_{0l} \ftp(u_{k}, u_{k+l})\\ 
			&	&+& \sum_{(l, m) \in N} \Er \Tr A_{lm} \ftp(u_{k + l}, u_{k+m})
		\end{aligned}.
		\end{equation}
		Due to condition \eqref{eq:Etaylorrestrict}  in the difference between two fluxes, the contributions from the last line cancel out each other. We further realize that
		\begin{equation}
		\begin{aligned}
			\sum_{l = -z}^{2p-z} A_{0l} + A_{l0} & \overset{\phantom{\eqref{eq:Etaylorrestrict}}}{=} \sum_{l, m = -z}^{2p-z} A_{lm} - \sum_{(l, m) \in N} \Er \Tr A_{lm} 
			\overset{\eqref{eq:Efconres}}{=} 1 -  \sum_{(l, m) \in N} \Er \Tr A_{lm} \\
			&\overset{\eqref{eq:Etaylorrestrict}}{=} 1 - \sum_{(l, m) \in N} B_{lm} 
			\overset{\eqref{eq:Efconres}}= \sum_{l, m = -z}^{2p-z} B_{lm} - \sum_{(l, m) \in N} B_{lm} = \sum_{l = -z}^{2p-z} B_{0l} + B_{l0}
		\end{aligned}
		\end{equation}
		holds due to \eqref{eq:Etaylorrestrict} and \eqref{eq:Efconres}. If we analyze the difference between both fluxes in the semi-discrete scheme formulation, we get 
		\begin{equation*} \label{eq:lcfdiff}
			\begin{aligned}
				\derd{u_k(t)}{t} =&  \frac{f_{k - \frac 1 2} - f_{k+ \frac 1 2}}{\Delta x} \\
				=& \frac{1}{\Delta x} \left( \sum_{l = -z}^{2p} A_{0l}\ftp(u_{k}, u_{k+l}) + A_{l0}\ftp(u_{k+l}, u_{k})
				- B_{0l}\ftp(u_{k}, u_{k+l}) - B_{l0}\ftp(u_{k+l}, u_{k}) \right)\\
				=& \frac{1}{\Delta x} \Bigg( \sum_{l = -z}^{2p} A_{0l}\ftp(u_{k}, u_{k+l}) + A_{l0}\ftp (u_{k+l}, u_{k}) - A_{0l}\ftp(u_{k}, u_{k}) - A_{l0}\ftp(u_{k}, u_{k}) \\ 
				&+ B_{0l}\ftp(u_{k}, u_{k}) + B_{l0}\ftp(u_{k}, u_{k}) - B_{0l}\ftp(u_{k}, u_{k+l}) - B_{l0}\ftp(u_{k+l}, u_{k}) \Bigg).
			\end{aligned}
		\end{equation*}
		Finally, by multiplying with entropy variable $\derd{U}{u}$  from the left previous, we get by elementary calculations 	
		\begin{equation}\label{eq:ecproof}
			\begin{aligned}
			&\derd{U(u_k)}{t}
			=\left\langle \derd{U}{u}(u_k(t)), \derd{u_k}{t} \right \rangle \\
			=&  \frac {1}{\Delta x} \Bigg(\sum_{l = -z}^{2p}\left \langle \derd{U}{u}(u_k), A_{0l}\ftp(u_{k}, u_{k+l}) - A_{0l}\ftp(u_{k}, u_{k}) \right \rangle 
			+  \sum_{l = -z}^{2p}\left \langle \derd{U}{u}(u_k),A_{l0}\ftp(u_{k+l}, u_{k}) - A_{l0}\ftp(u_{k}, u_{k}) \right \rangle\\
			&+ \sum_{l = -z}^{2p}\left \langle \derd{U}{u}(u_k), B_{0l}\ftp(u_{k}, u_{k}) - B_{0l}\ftp(u_{k}, u_{k+l}) \right \rangle 
			+ \sum_{l = -z}^{2p} \left \langle \derd{U}{u}(u_k),  B_{l0}\ftp(u_{k}, u_{k}) - B_{l0}\ftp(u_{k+l}, u_{k}) \right \rangle \Bigg)\\
			=& \frac{1}{\Delta x}\sum_{l = -z}^{2p} \Big( A_{0l}\Ftp(u_{k}, u_{k+l}) - A_{0l}\Ftp(u_{k}, u_{k}) + A_{l0}\Ftp(u_{k+l}, u_{k}) - A_{l0}\Ftp(u_{k}, u_{k}) \\ 
			&  + B_{0l}\Ftp(u_{k}, u_{k}) - B_{0l}\Ftp(u_{k}, u_{k+l}) + B_{l0}\Ftp(u_{k}, u_{k}) - B_{l0}\Ftp(u_{k+l}, u_{k}) \Big ) \\
			=&\frac{1}{\Delta x}\sum_{l = -z}^{2p} \Big( A_{0l}\Ftp(u_{k}, u_{k+l}) + A_{l0}\Ftp(u_{k+l}, u_{k}) 
			 - B_{0l}\Ftp(u_{k}, u_{k+l})- B_{l0}\Ftp(u_{k+l}, u_{k}) \Big )\\
			=& \frac{F_{k-\frac 1 2} - F_{k + \frac 1 2}}{\Delta x}.
			\end{aligned}
		\end{equation}
		Here, we used that we can swap all two-point fluxes for the entropy fluxes $\Ftp(u_l,u_m) - \Ftp(u_m,u_r) = \skp{\derd{U}{u}, \ftp(u_l, u_m) - \ftp(u_m, u_r)}$. These are in turn equal to the difference of the entropy fluxes as given in equation \eqref{eq:EFansatzA} and \eqref{eq:EFansatzB}. Because of condition \eqref{eq:Elcfres} and \eqref{eq:Etaylorrestrict}, all four fluxes are clearly linear combined two-point fluxes in the sense of Definition \ref{def:lcflux}. Entropy conservation is proven as the last equation is the sought after per cell entropy equality.  
	\end{proof}
\end{theorem}
The two results given before can be used to construct entropy conservative high-order fluxes near the boundary of a domain. The general approach is to find appropriate linear combinations. The stencil does not have to be a centered stencil as in  \cite{lefloch2002fully}. It can be a stencil that is shifted to the left or right of the edges where we would like to determine a numerical flux. This is also pictured in Figure \ref{fig:FluxStencil}.
The approach can be summarized as follows: \\

 One starts by using a standard flux like the one proposed by LeFloch, Mercier and Rhode for inner cell interfaces. If a $2p$ order accurate flux with a $2p$ stencil is used, this flux will be suitable up to interface $N-p+\frac 1 2$ if $N$ cells are used. The calculation of $\derd{u_k}{t}$ uses a stencil of $2p+1$ points for an inner cell and this is why we will assume that our outer cells use also a $2p+1$ point wide stencil. We set $k = N-p$ and add a ghost cell filled with the boundary value on the right side with number $N+1$. One concludes that the fluxes $f_{k+1+\frac 1 2}, \dots, f_{k + p + \frac 1 2}$ should all use the same $2p+1$ wide stencil whose right most point is the ghost cell $N+1$. One therefore uses Lemma \ref{lem:HOBF} to construct $f_{k+m+\frac 1 2}$ out of $f_{k+m-\frac 1 2}$ for $m \in \sset{1, \dots, p}$. This approach is given a formal description in the following theorem. 
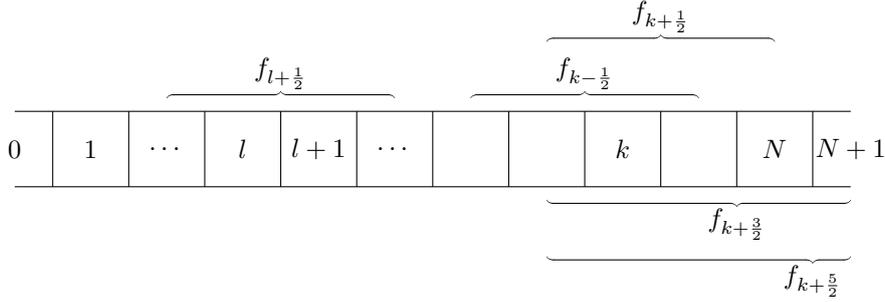
\begin{figure}
\centering
	\begin{tikzpicture}
		\draw(-5.5, 0) -- (5.5, 0);
		\draw(-5.5, 1) -- (5.5, 1);
		
		\draw(-5, 0) -- (-5, 1);
		\draw (-4, 0) -- (-4, 1);
		\draw (-3, 0) -- (-3, 1);
		\draw (-2, 0) -- (-2, 1);
		\draw (-1, 0) -- (-1, 1);
		\draw (0, 0) -- (0, 1);
		\draw (1, 0) -- (1, 1);
		\draw (2, 0) -- (2, 1);
		\draw (3, 0) -- (3, 1);
		\draw (4, 0) -- (4, 1);
		\draw (5, 0) -- (5, 1);
		
		\draw (-5.5, 0.5) node {$0$};
		\draw (-4.5, 0.5) node {$1$};
		\draw (4.5, 0.5) node {$N$};
		\draw (5.5, 0.5) node {$N+1$};
		\draw (2.5, 0.5) node {$k$};
		
		\draw (-2.0, 1.5) node {$f_{l+\frac 1 2}$};
		\draw (-2.5, 0.5) node {$l$};
		\draw (-1.5, 0.5) node{$l+1$};
		\draw [decorate, 
		decoration = {calligraphic brace,
			raise=5pt,
			aspect=0.5}] (-3.5,1) --  (-0.5,1);
		
		\draw (-3.5, 0.5) node {$\dots$};
		\draw (-0.5, 0.5) node {$\dots$};
		
		\draw (2.0, 1.5) node {$f_{k-\frac 1 2}$};
		\draw [decorate, 
		decoration = {calligraphic brace,
			raise=5pt,
			aspect=0.5}] (0.5,1) --  (3.5,1);
		\draw (3.0, 2.25) node {$f_{k+\frac 1 2}$};
		\draw [decorate, 
		decoration = {calligraphic brace,
			raise=5pt,
			aspect=0.5}] (1.5,1.75) --  (4.5,1.75);

		\draw (4.0, -0.5) node {$f_{k+\frac 3 2}$};
		\draw [decorate, 
		decoration = {calligraphic brace,
			raise=5pt,
			aspect=0.375}] (5.5,0.0) --  (1.5,0.0);
		
		\draw (5.0, -1.25) node {$f_{k+\frac 5 2}$};
		\draw [decorate, 
		decoration = {calligraphic brace,
			raise=5pt,
			aspect=0.125}] (5.5,-0.75) --  (1.5,-0.75);
	\end{tikzpicture}
	\caption{Depiction of the domain. The cells on the boundary, the inter cell fluxes and their respective stencils are drawn for the right cell boundary. The inner flux is the $4th$ order accurate flux consisting of the linear combination given by LeFloch, Mercier and Rhode applied to a symmetric entropy conservative flux. }
	\label{fig:FluxStencil}
\end{figure}

\begin{theorem}[Boundary aware entropy conservative fluxes]
	There exist entropy conservative fluxes of order $q \leq 2p-1$ suitable for usage with the inner fluxes from Theorem \ref{thm:LMR}.
	\begin{proof}
		The proof uses the results from Lemma \ref{lem:HOBF} and Theorem \ref{lem:ECF}. Given a domain with $N$ inner cells we can use for inner cells, as depicted in Figure \ref{fig:FluxStencil}, the usual centered linear combination from  \cite{lefloch2002fully}. The last interface using this flux to the right is $k + m  - \frac 1 2$ with $k = N-p$ and $m = 1$. We now construct from this flux, the flux on the next cell interface near the boundary by
		\begin{enumerate}
			\item Artificially enlarging the stencil to the right using $\Er$, i.e. still the same flux is represented by $\Er A$.
			\item Convert the matrix representation $\Er A$ of $f_{k + m - \frac 1 2}$ from Lemma \ref{lem:matrixflux} to be based around $ k + m + \frac 1 2$ using the operator $\Tr$, i.e. the same flux is now represented by $\Tr \Er A$.
			\item Use Lemma \ref{lem:HOBF} to calculate a high-order flux for $x_{k + \frac 1 2}$ given by the matrix $B$. This boils down to calculating the vector $v$, giving the difference between both matrices.
		\end{enumerate}
		This procedure is repeated from point $2$ with $B$ in place of $\Er A$ and $m := m+1$ until the last flux was constructed. Clearly for all fluxes holds Theorem \ref{lem:ECF} implying a per cell entropy inequality. 
	\end{proof}
\end{theorem}

The coefficient matrices in the theorem above could be greatly simplified, but we are keeping them this way. It is therefore possible to couple this linear combination with as many other high-order fluxes in the middle as possible. All needed restrictions were instead given as restrictions on the structure of the matrices.

\begin{remark}[Order of Accuracy at the Boundary]
	Even if the inner order is $2p$ an outer order of accuracy of $2p-1$ is high enough to ensure a global order of $2p$, as for summation-by-parts operators \cite{Gustafsson1975Convergence, Gustafsson1981Convergence}. This will be also demonstrated in our numerical experiments.
\end{remark}

\subsection{Inflow, outflow and reflective boundary conditions using boundary aware fluxes} \label{ssec:bcs}
We will now outline how these boundary aware fluxes can be used to design boundary conditions for boundary types that are relevant for simulations of compressible inviscid flow, i.e. inflow, outflow and reflective boundary conditions \cite{Toro1997Riemann}.
\begin{enumerate}
	\item Inflow boundaries:	If an inflow boundary should be implemented, we  add a ghost cell outside of the domain, as in figure \ref{fig:FluxStencil}, and copy the boundary value into this inflow cell. The fluxes next to this interface should be the boundary fluxes that are designed in this publication.
	\item Outflow boundaries: If an outflow boundary should be implemented, we  add a ghost cell outside of the domain, as in figure \ref{fig:FluxStencil}, and add in the value of the last cell in the grid into this ghost cell. The fluxes next to this interface should be the boundary fluxes that are designed in this publication.
	\item Reflective boundaries: For a reflective boundary condition the most important property is the symmetry of the boundary condition - incoming particles are reflected of the boundary as if they came from the other side of the boundary with the same tangential speed, but opposite normal speed as the incoming particle before the reflection. Therefore the problem is equivalent to an enlargement of the domain, and the cells on the opposite of the boundary are initialized with the same pressure and density, but appropriately mirrored speeds. In this case, the fluxes for an interior cell interface of the domain are used.
\end{enumerate}

%% file: 3_Disflux.tex
\subsection{Entropy dissipative fluxes}\label{se_entropy_dissipative}
After the development  of high-order entropy conservative fluxes  that can also be used at non-periodic boundaries, entropy dissipative fluxes will be designed before we focus on the convergence properties of our FV/FD schemes. To construct such fluxes, we need an additional generalization of the linear combination defined in Definition \ref{def:lcflux}.  
\begin{definition}[General linear combined flux]\label{de_general_flux}
	Let $\fnum_{l, m}(u_L,u_R)$ be a family of two-point fluxes and $A  \in \R^{\iset}$  a matrix
	satisfying $ A_{ll} = 0, \forall  l \in Z = \{-z+1, \dots, 2p-z \} $.
	A numerical flux given by 
	\begin{equation}
		f(u_{k-z+1}, \dots, u_{k+2p-z}) := \sum_{l, m = -z+1}^{2p-z} A_{lm} \fnum_{l, m}(u_{k+l}, u_{k+m})
		\label{eq_dcombination}	
	\end{equation}
	is called  general combined  numerical flux with stencil $k + Z = \{k-z+1, \dots, k+2p-z\}$. 
	\end{definition}
 The value $z$ shifts the stencil of the designed  flux. We obtain for  $z = p$ a  symmetric stencil. As before, the set $Z$ holds the offsets of the arguments for the flux from the grid point $k$. Therefore, it holds $Z \subset \Z$. The difference between the linear combinations \ref{def:lcflux} and \ref{de_general_flux} lies in the selection of the baseline fluxes. In above definition, the flux can depend on the indices. Using a combination of dissipative and entropy conservative two-point fluxes allows us to steer the dissipation of the combined flux.
 
\begin{lemma}
Let $\ftp$ be an entropy conservative, symmetric and smooth two-point flux and let further $A, B$ be two matrices inducing an entropy conservative flux combination as designed in Theorem \ref{lem:ECF} for $\ftp$.  Additionally, let $g$ be an entropy dissipative two-point flux with entropy flux $G$ and $\alpha \in [0,1]$  and the matrices $A$ and $B$ satisfy
\[
\begin{aligned}
	A_{0,1} \geq 0, \quad A_{1,0} \geq 0,  \quad	B_{0,1} \geq 0, \quad B_{1,0} \geq 0.\\
	\end{aligned}
\]
Using the flux family 	\begin{equation}\label{eq_applied_flux}
		f_{l,m, \alpha}(u_L, u_R) = \begin{cases} \alpha g(u_L, u_R) + (1-\alpha) h(u_L, u_R)&  (m, l) = (0, 1)    \\
																	\alpha g(u_R, u_L) + (1-\alpha) h(u_L, u_R) &  (m, l) = (1, 0)  \\
																	h(u_L, u_R) & \text{else} \\
																	\end{cases}
	\end{equation}	
	as our base flux family in  \eqref{eq_dcombination} results in a  general linear combined flux \eqref{eq_dcombination} that
	is entropy dissipative with a numerical entropy flux defined through the family of two-point entropy fluxes
	\[
			F_{l,m, \alpha}(u_L, u_R) = \begin{cases} \alpha G(u_L, u_R) + (1-\alpha) H(u_L, u_R),&  (m, l) = (0, 1)   \\
			\alpha G(u_R, u_L) + (1-\alpha) H(u_L, u_R), &  (m, l) = (1, 0) \\
			H(u_L, u_R), & \text{else.} \\
		\end{cases}
	\]
	\begin{proof}
	The proof follows the steps of the proof of Theorem  \ref{lem:ECF}. Instead of the equality in  equation \eqref{eq:ecproof}, we obtain the inequalities where the entropy dissipative fluxes are entered for the products of entropy variables and fluxes. This is possible due to 
		\[
			\skp{\derd U u(u_m), f_{l, m, \alpha}(u_l, u_m) - f_{l, m, \alpha}(u_m, u_m)} = \skp{\derd U u, f_{l, m, \alpha}} \leq F_{l, m, \alpha}(u_l, u_m) - F_{l, m, \alpha}(u_l, u_m) .
		\]
		Consult also \cite{klein2022using}[Theorem 3.1]. The rest follows analogously. 
		\end{proof}
	\end{lemma}

The application of entropy dissipative fluxes results in entropy dissipative FV/FD schemes. These are needed if discontinuous solutions will be approximated. 
	
	\begin{remark}
		The condition on the positivity of the matrix elements of $A$ and $B$ in the previous theorem is needed to ensure the dissipativity of the constructed flux. A negative entry in the corresponding entries would produce antidissipative effects in the flux $f_{k + \frac 1 2}$, even if a dissipative $f(u_L, u_R)$ is used. All matrices $A$ used for the construction of linear combined fluxes in this publication have this property, as can be seen in the appendix.
		\end{remark}

	

%% file: 4_lwtheorem.tex
\section{Proof of the Lax-Wendroff theorem} \label{se:lwthm}
Our schemes were built in the classical FV/FD setting and it is therefore not astonishing that one is able to prove a Lax-Wendroff theorem for our schemes. 
The validation of a Lax-Wendroff theorem ensures that our approximated solutions converge to the weak entropy solution under additional compactness assumptions. Rather than demonstrating directly a Lax-Wendroff theorem for our schemes, 
we follow a different approach. 
In  \cite{shi2018local}, a rigorous definition for local conservation is given and sufficient criterias are named. If such conditions are fulfilled 
a Lax-Wendroff theorem follows directly. Therefore, we only check such criterias to ensure Lax-Wendroff  \cite{LW1960}.
Following  \cite{shi2018local}, we define local conservation of  a numerical scheme 
\begin{equation}\label{scheme_gen}
\frac{u_h(\cdot, t^{n+1})-u(\cdot, t^n)} {\Delta t^n}=L(u_h(\cdot, t^n))
\end{equation}
with space discretization $L$, grid size $h$,  numerical solution $u_h$, and time discretization $0=t^0<t^1 <\cdots$ with $t^n\to \infty$ as $n\to \infty$ and time step $\Delta t_n=t^{n+1}-t^n$ as follows:
\begin{definition}\label{def_conservative}
A numerical scheme of the form \eqref{scheme_gen} is locally conservative if there are conserved quantities and fluxes, both of which locally depend on the numerical solution and satisfy: 
\begin{enumerate}
\item They fulfill additionally: 
\begin{equation}\label{scheme_gen_2}
\frac{\bar u_k^{n+1} -\bar u_k^{n} } {\Delta t^n} +\frac{1}{|I_k|} \left(g_{k+\frac 1 2}(u^t_h) -g_{k-\frac 1 2}(u^t_h)  \right)=0,
\end{equation}
where $\bar u_k^{n}$ are generalized locally conserved quantities and $g_{k\pm\frac 12}$ are generalized fluxes. They depend on $u_h^n(B^c_k)$ with $B^c_k=\left\{ x\in \R| |x-w_k|< ch\right\}$ where $w_k$ is the midpoint of the interval $I_k$, and $c \geq 1$ is independent of the mesh size $h$. 
\item \emph{Consistency:} If $u_h^n(x) \equiv u$ constant $\forall x \in B^c_k$, we have $\bar u_k^{n+1}=u$ and $g_{k+\frac 1 2}\left(u^t_h \right)= f(u)$.
 \item \emph{Boundedness:} We get 
 \begin{equation}
 \begin{aligned}
 \left| \bar u_k^{n}- \bar v_k^{n}   \right| \leq C \norm{u_h^n-v_h^n}_{L^{\infty}(B_k)},\\
  \left|g_{k+\frac 1 2}\left(u^t_h \right)- g_{k+\frac 1 2}\left(v^t_h \right)  \right| \leq C \norm{u_h^n-v_h^n}_{L^{\infty}(B_k)},
 \end{aligned}
 \end{equation}
 for two functions ($u_h$ and $v_h$) in the numerical solution space. 
 \end{enumerate}
\end{definition}
It was proven in  \cite{shi2018local}, that a Lax-Wendroff type theorem is valued for local conservative methods\footnote{If additionally the numerical solutions are TV bounded. }. Therefore, we have only to check the conditions of Definition \ref{def_conservative}
for our discrete scheme from Section \ref{se:bdflux}. However, the notation differs slightly and we give the translation rules in the following part. 

Our conserved quantities are the values of $u$ at the middle points of the cells that can be also interpreted as mean values of the solution in the cell. It is 
$
	\bar u^n_h = u^n_k.
$
We will use our numerical fluxes as cell interface fluxes 
\[
	g_{k+\frac 1 2}\left(u^t_h\right) := \fnum_{{k+\frac 1 2}}\left(u_{k-p+1}, \dots, u_{k+p}\right)
\]
and their domain of dependence $B_k^c$ is given by $B_k^c = \{x \in \R| |x - w_k| \leq ch \}$ with  $c =2 p$, where $2p$ is the stencil radius of our flux, i.e. the biggest width on which our high-order flux depends. Conservation is satisfied by definition since our scheme was designed in conservative form. 
What remains is the boundedness of the quantities and fluxes. 
Assume there are two approximate solutions $u^t_h$ and $v^t_h$ given satisfying $\norm{u^t_h - v^t_h}_{L^\infty} \leq \infty$. 
The mean values are clearly bounded by 
\[
	\abs{\bar u^t_k - \bar v^t_k} \leq \frac{1}{\mu(I_k)}\int_{I_k} \abs{u^t_h(x) - u^t_h(x)} \intd x  \leq \frac{\mu(I)}{\mu(I)} \norm{u^t_h - v^t_h}_{L^\infty} \leq \infty, 
\] where $\mu(I_k)$ is the Lebesgue measure of interval $I_k$.

To proof the boundedness of the fluxes, we use the Lipschitz continuity
and the Lipschitz bound $L_{GT}$ of the underlying convex combined baseline flux \eqref{eq_applied_flux}, cf. \cite[Theorem 3.1]{klein2022using}. This assures us, that every flux in our family of fluxes $f_{l, m}$ is Lipschitz continuous, where $L_{l, m}$ shall be the Lipschitz constant for flux $(l, m)$ of our family.
Let now $k$ be fixed. Then follows
\begin{equation}
\begin{aligned}
	\abs{g_{k+\frac 1 2}(u^t_h) - g_{k+\frac 1 2}(v^t_h)} 
	& \leq \sum_{l, m} \abs{A_{l, m}}\norm{f_{l, m, \alpha}(u^n_{k+l}, u^n_{k+m}) - f_{l,m, \alpha}(v^n_{k+l}, v^n_{k+m})} \\
	&\leq  \sum_{l, m} \abs{A_{l, m}} L_{l, m} (2p+1)  \norm{u^t_h - v^t_h}_\infty.
\end{aligned}
\end{equation}
 Please note that we used the equivalence of norms here, as only $2p+1$ arguments are entered into our fluxes and therefore all norms on the space of their arguments are equivalent. 
Therefore, all the 
conditions of Definition  \ref{def_conservative} are fulfilled by our entropy conservative/dissipative FV schemes and we have: 

\begin{theorem}[Lax-Wendroff]
	Assume the  numerical solution $u_h^n$ of our discrete scheme satisfies
	\[
		h \sum_k \max_{x \in B_k}\abs{u^n_h(x) - u^n_h(w_j)} \to 0, \text{ as } h \to 0, \quad \forall n \geq 0.
	\]
	If $u_h$ converges uniformly almost everywhere to some function $u$ as $\Delta t$, $h \to 0$, then  $u$ is  a weak solution of the conservation law \eqref{eq_con}. 
\end{theorem}

%% file: 5_NT.tex
\section{Numerical Tests} \label{se_NT}
\subsection{Numerical Tests of the Boundary Fluxes}

\begin{figure}[H]
	\begin{subfigure}{0.45\textwidth}
		\includegraphics[width=\textwidth]{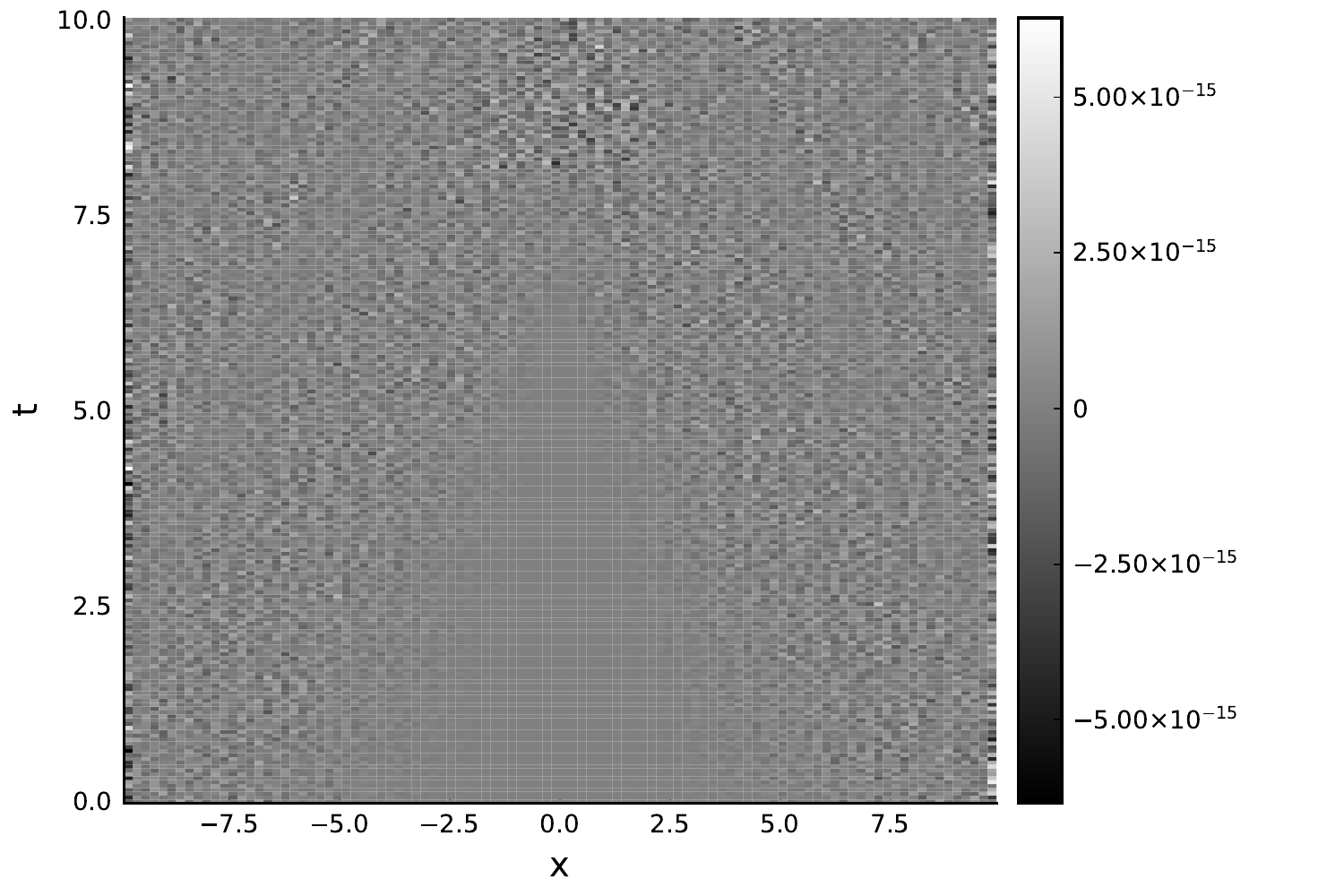}
		\caption{Per cell entropy equality violation up to $T=10$.}
	\end{subfigure}
	\hfill
	\begin{subfigure}{0.45\textwidth}
		\includegraphics[width=\textwidth]{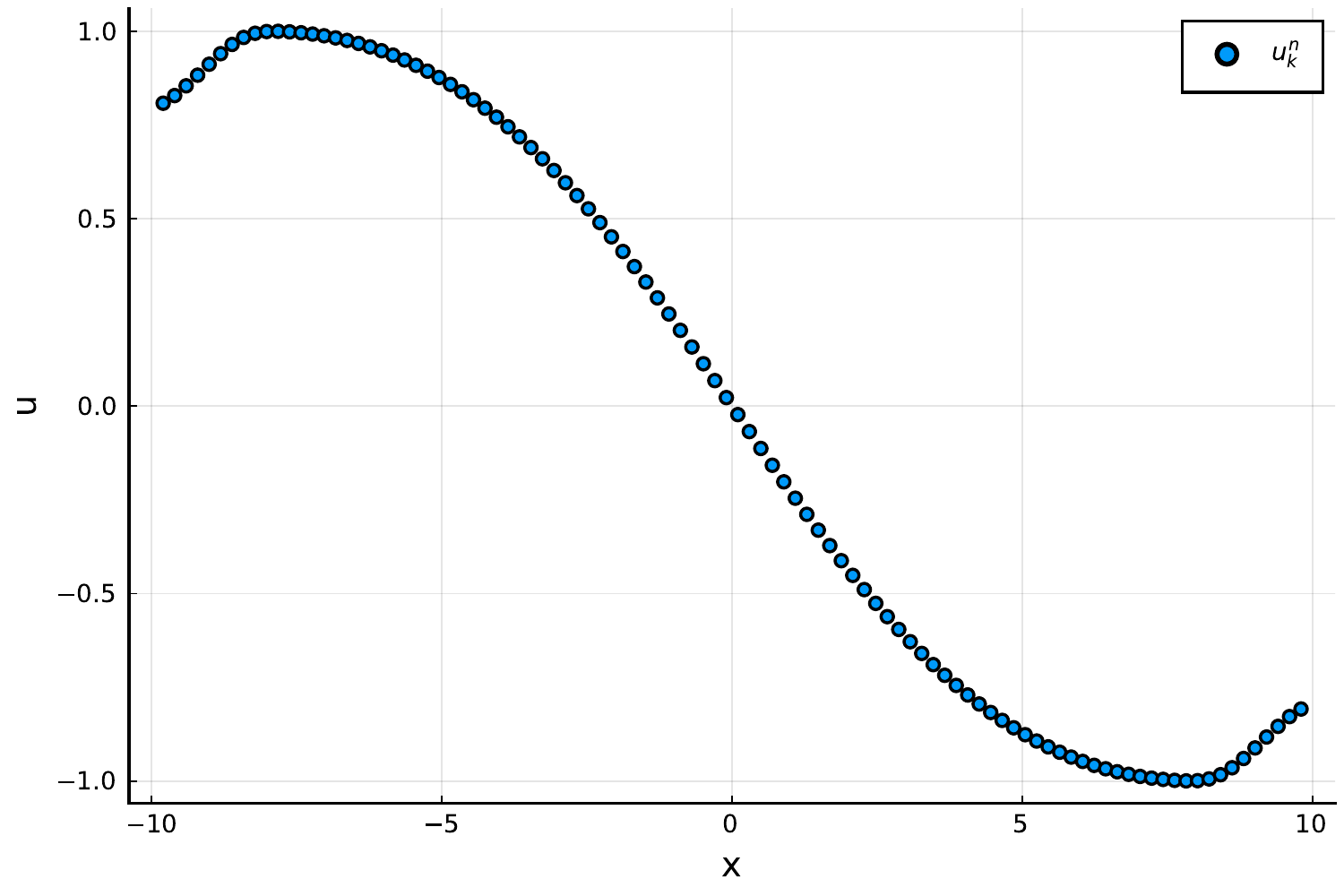}
		\caption{Solution at $T = 2$}
	\end{subfigure}
	\begin{subfigure}{0.45\textwidth}
		\includegraphics[width=\textwidth]{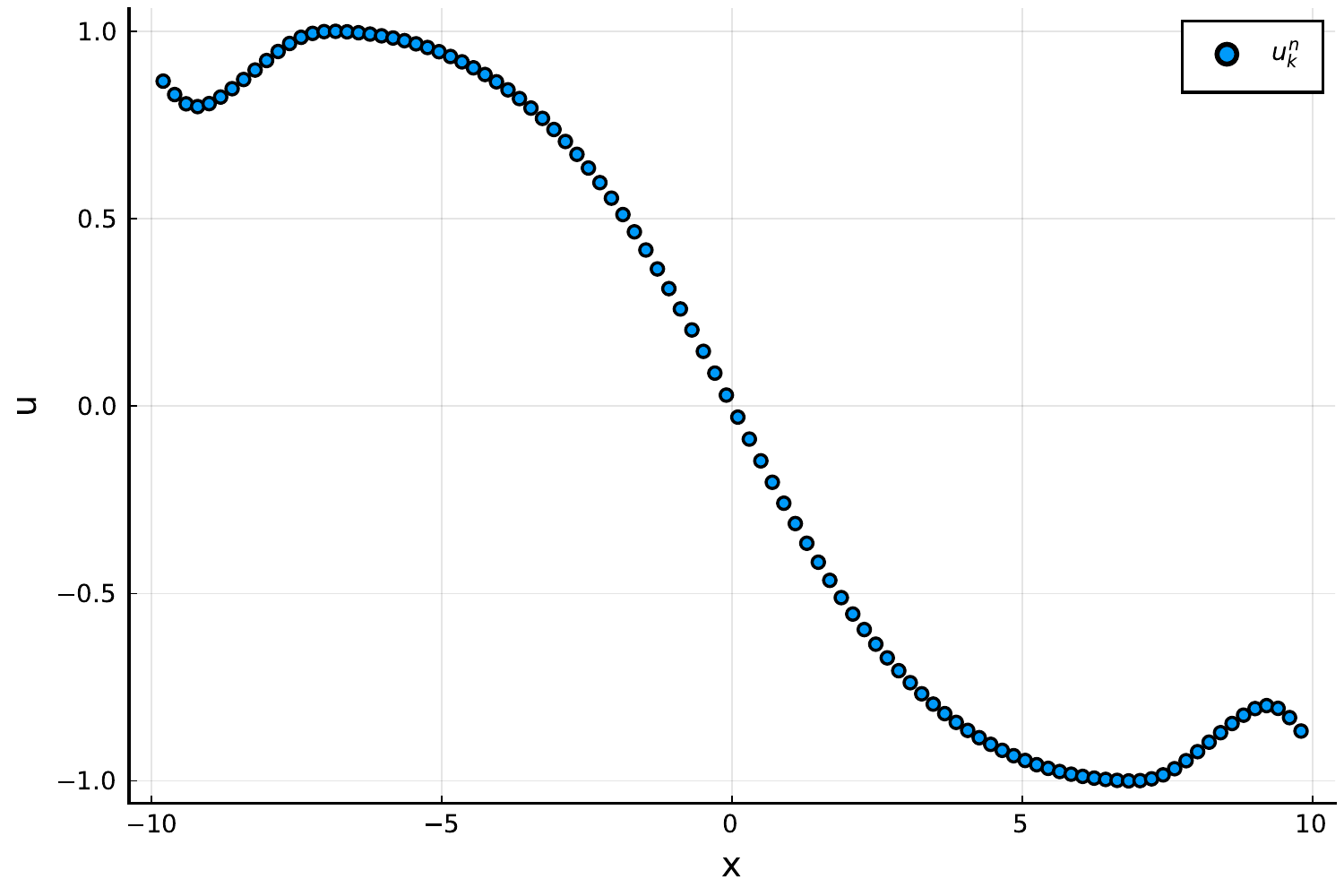}
		\caption{Solution at $T = 3$}
	\end{subfigure}
	\hfill
	\begin{subfigure}{0.45\textwidth}
		\includegraphics[width=\textwidth]{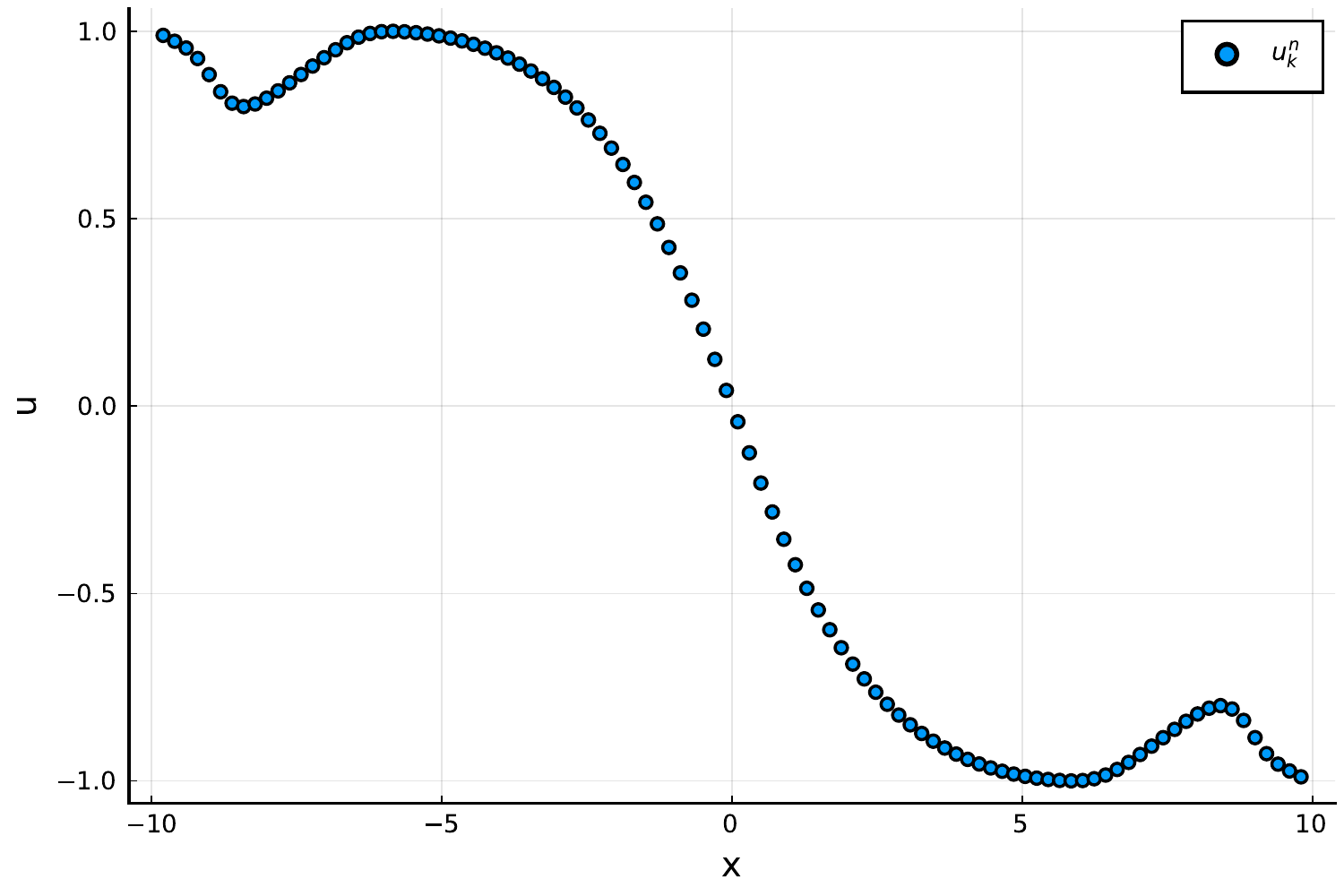}
		\caption{Solution at $T = 4$}
	\end{subfigure}
	\begin{subfigure}{0.45\textwidth}
		\includegraphics[width=\textwidth]{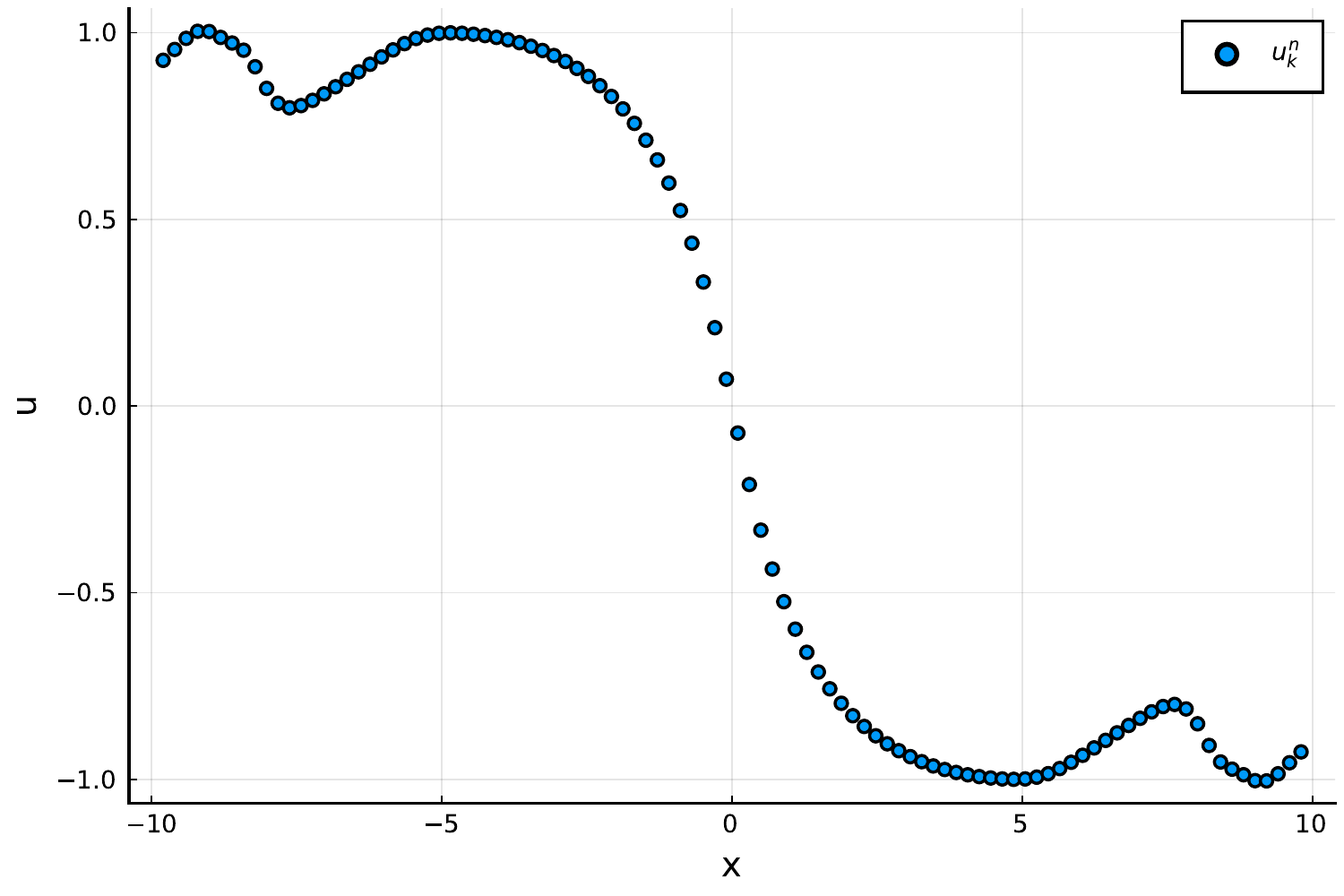}
		\caption{Solution at $T = 5$}
	\end{subfigure}
	\hfill
	\begin{subfigure}{0.45\textwidth}
		\includegraphics[width=\textwidth]{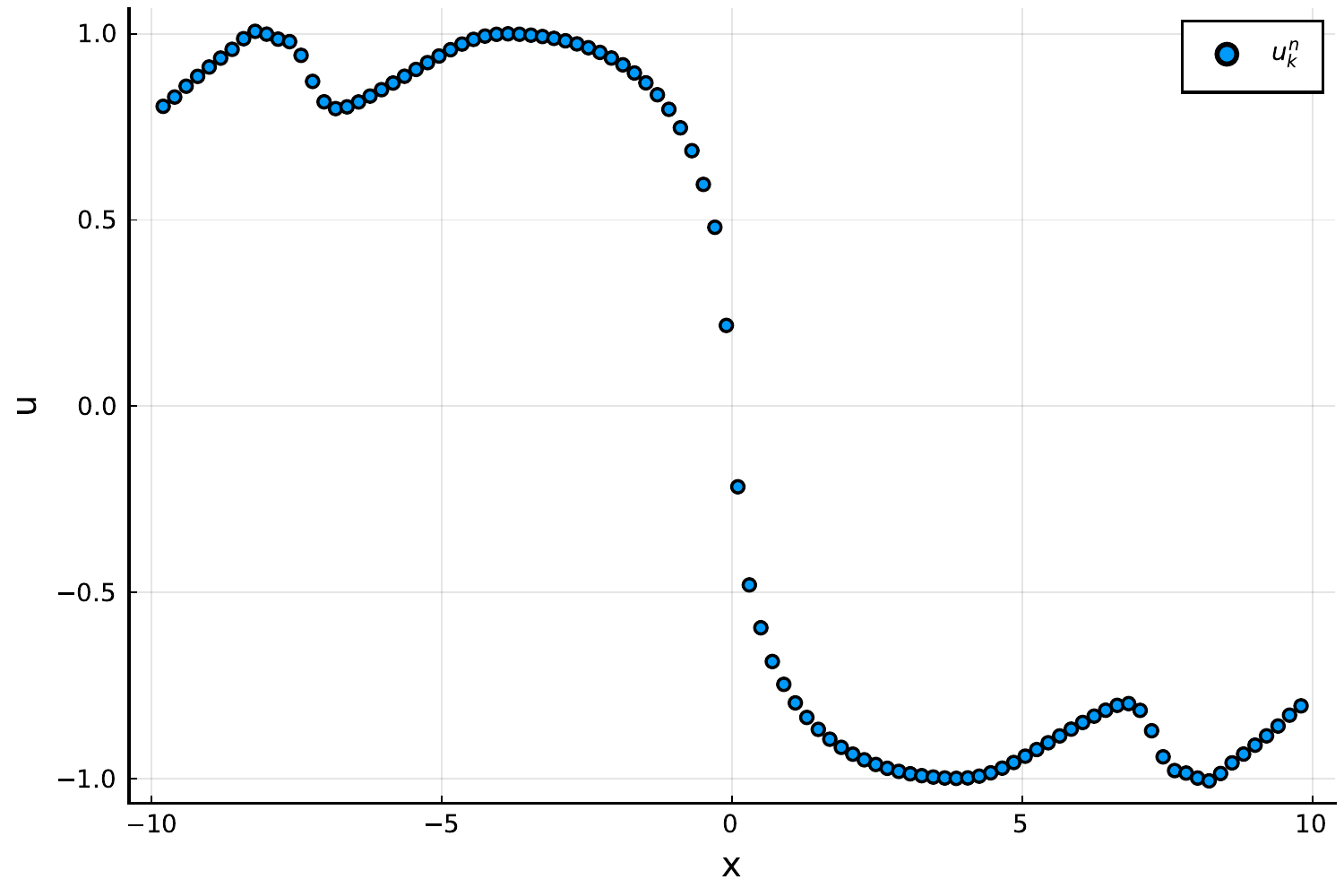}
		\caption{Solution at $T = 6$}
	\end{subfigure}
	\caption{Solution to the Burgers' equation using the entropy conservative flux of Tadmor augmented with the centered and non-centered linear combination from Section \ref{se:bdflux} to produce an $2p$ order accurate numerical flux. In this case $p=3$ was used. The entropy equality is satisfied up to a tolerance of $10^{-14}$. This error can be attributed to floating point rounding.}
	\label{fig:BCsol}
\end{figure}
\begin{figure}[H]
	\begin{subfigure}{0.48\textwidth}
		\includegraphics[width=\textwidth]{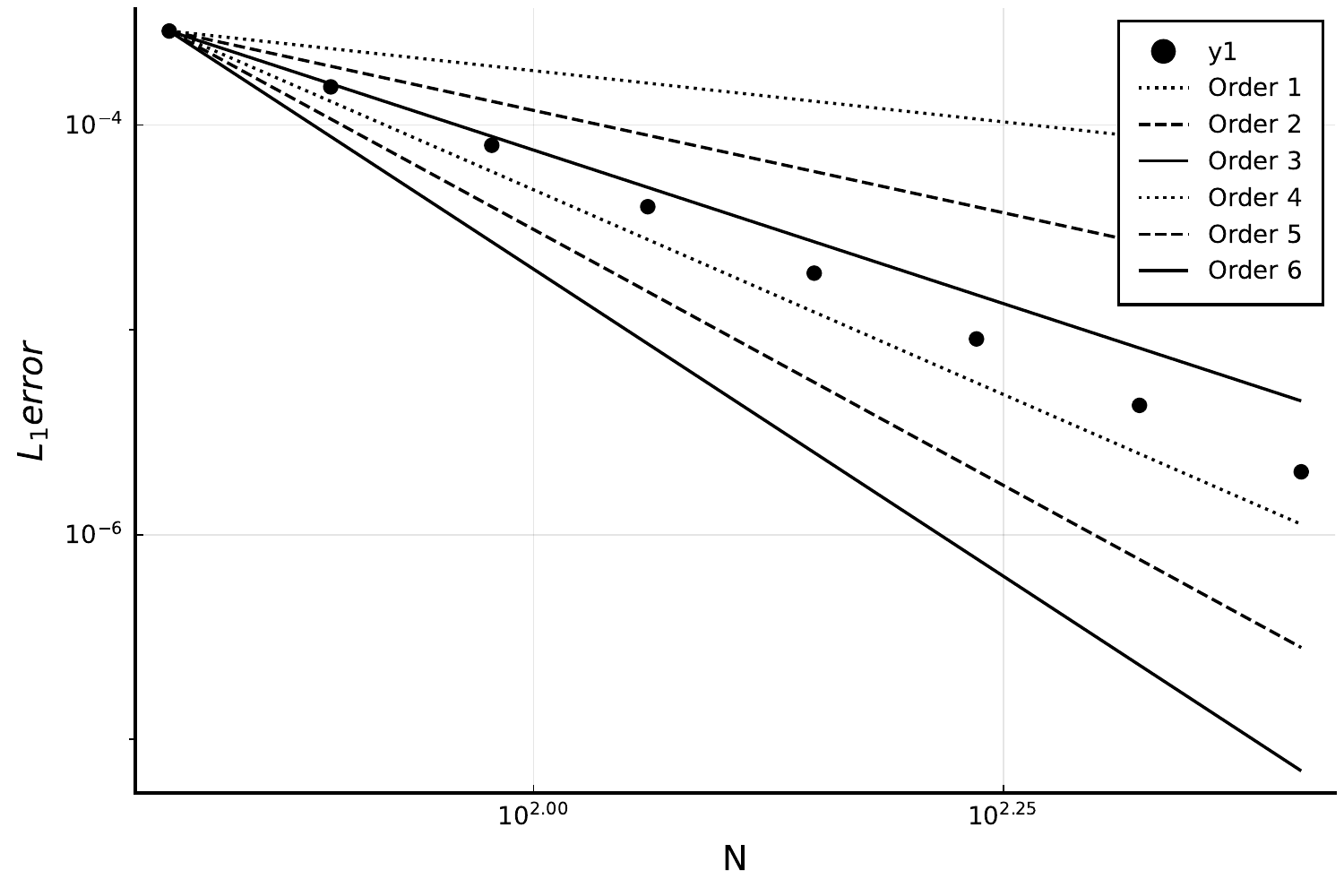}
		\subcaption{Convergence analysis for $p=2$}
	\end{subfigure}
	\hfill
	\begin{subfigure}{0.48\textwidth}
		\includegraphics[width=\textwidth]{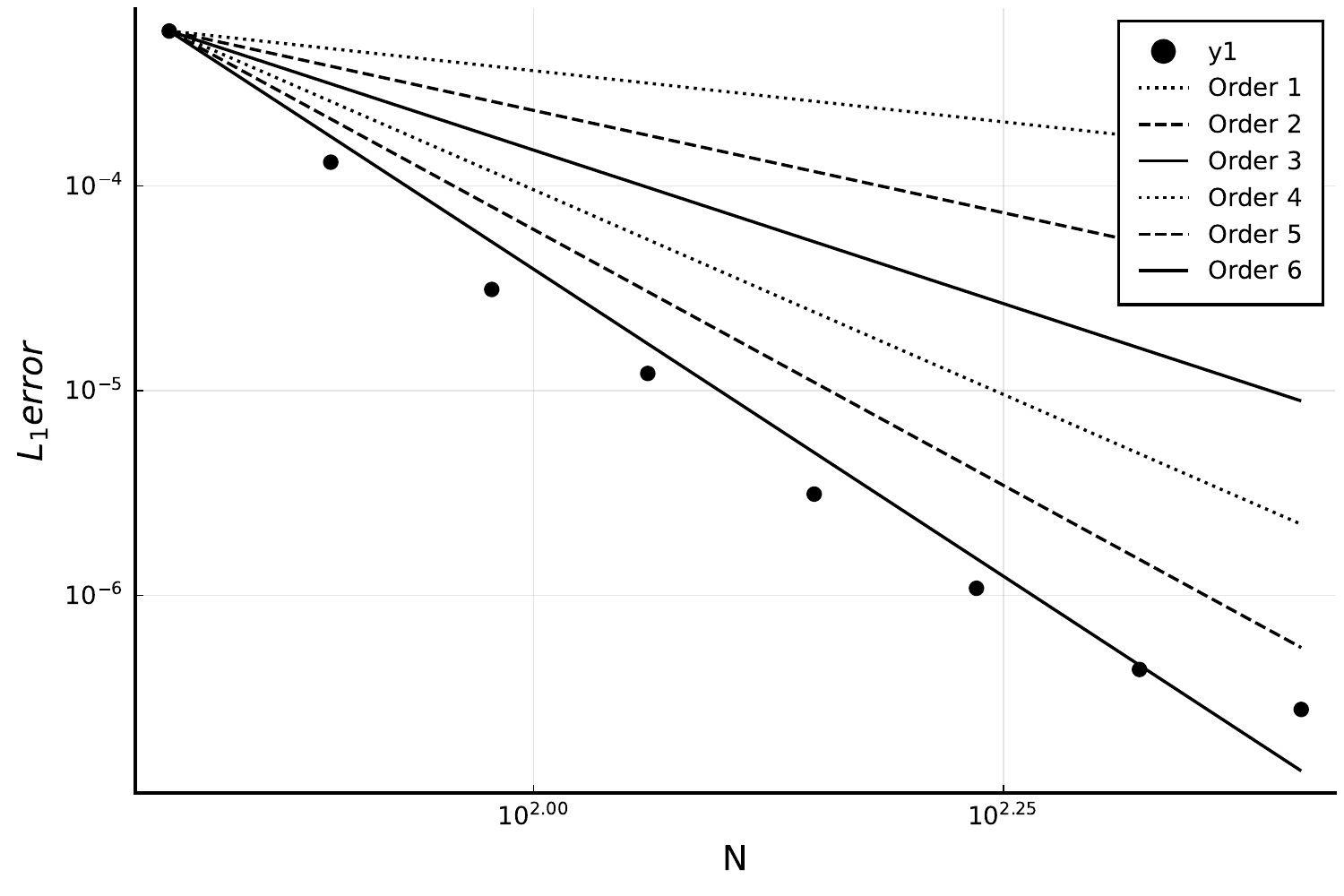}
		\subcaption{Convergence analysis for $p=3$}
	\end{subfigure}
	\caption{Convergence analysis for the high order entropy conservative boundary attached fluxes of section \ref{se:bdflux}. One expects an $2p$-order convergence rate from Theorem \ref{lem:HOBF} and this order is also demonstrated.}
\end{figure}
In Section \ref{se:bdflux},  we provided a novel approach to construct entropy conservative numerical fluxes of global order $2p$ by generalizing the linear combinations done by LeFloch, Mercier and Rhode. These fluxes are designed to handle, in particular, problems with non-periodic boundaries. 
We will now test if these fluxes are able to solve problems with the predicted order of accuracy. As no means of dissipation are present in these fluxes it is imperative to present a smooth problem with matching boundary conditions for this demonstration. The domain $\Omega = [-10, 10]$ was chosen with the initial and boundary conditions
\[
u(x, 0) = \sin\left( \frac {-\pi x}{20} \right), \quad u_l(t) = \frac{9}{10} + \frac{\cos\left(\frac{\pi t}{2}\right)}{10}, \quad  u_r(t) = -\frac{9}{10} - \frac{\cos\left(\frac{\pi t}{2}\right)}{10}
\]
for Burgers' equation
\[
\partial_t u + \partial_x \frac{u^2}{2}= 0.
\]
Solving these up to $T = 6$ with the presented fluxes for all non-inner points leads to the graphs in Figure \ref{fig:BCsol}. Time integration was carried out by the SSRPK(3,3) method and a grid of 100 points with a CFL number of $\lambda = 0.25$ was used. The results look satisfactory, but this was awaited already for a smooth problem as our fluxes are consistent and entropy conservative. 
The achieved experimental order of accuracy will be tested next. For this purpose 
\[
u(x, 0) = 1, \quad u_l(t) = 1 + \frac{\mathrm{e}^{-(t-5)^2}}{50}, \quad  u_r(t) = 1.0
\] 
was solved using $8$ exponentially spaced values for $N$ in the interval $64$ to  $256$. These solutions were compared to a reference solution calculated by an ENO scheme of order $2$ on a fine grid with $N = 16384$ cells. The ENO scheme \cite{ENOIII, SO1988, SO1989, SonarENO} used the Lax-Friedrichs flux and time integration was carried out using the SSPRK(3, 3) method.  
The CFL number was set to $\lambda = 0.5 \frac{64}{N}$ as this allows us to study the convergence up to order $6$ without problems stemming from the only third order convergence of the time stepping method. The theoretically predicted orders of convergence of $4$ and $6$ are demonstrated in the experiment for $p=2$ and $p=3$. 
\begin{figure}[H]
	\begin{subfigure}{0.48\textwidth}
		\includegraphics[width=\textwidth]{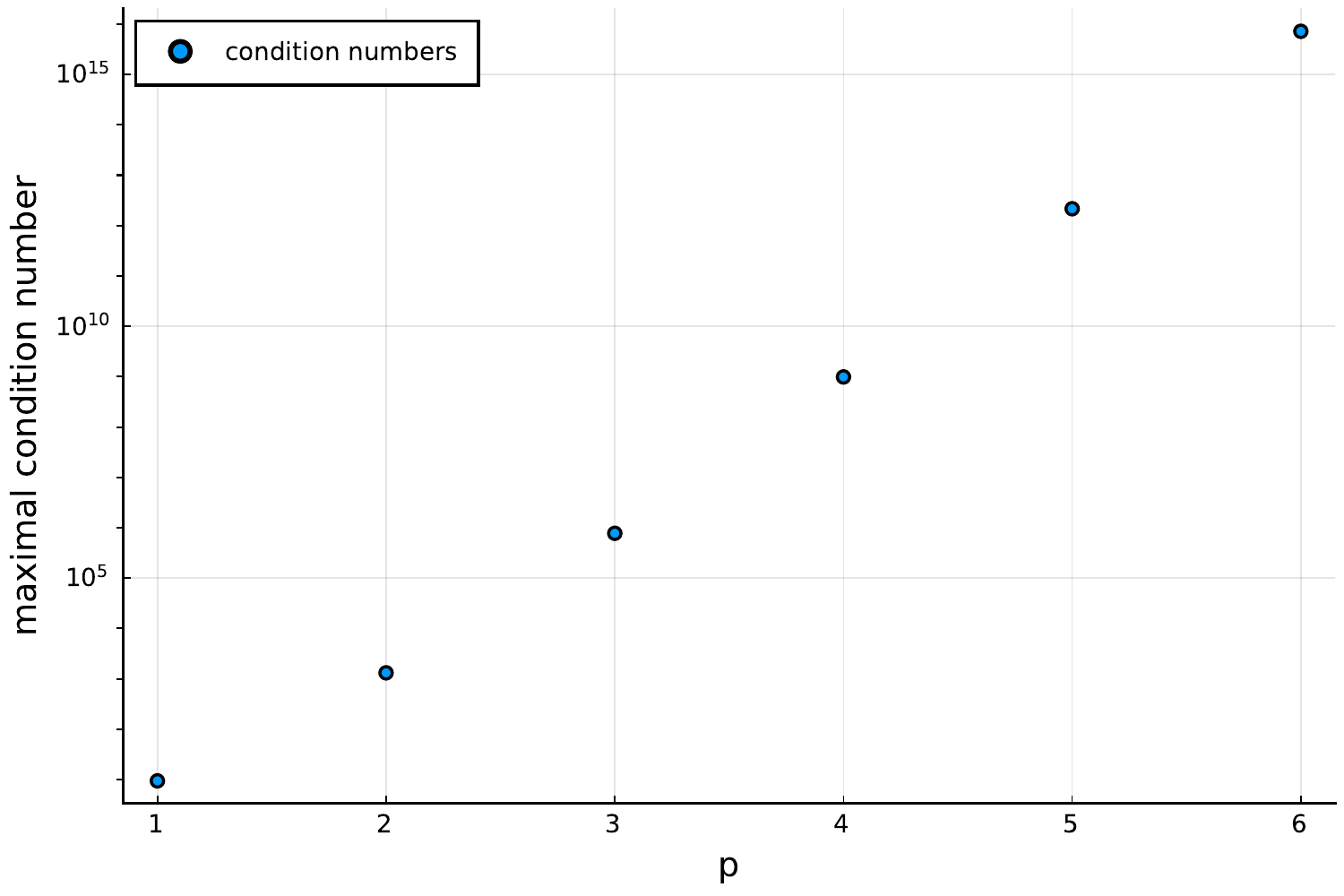}
		\subcaption{Maximal condition numbers for given $p$ used in the construction of the difference matrices of theorem \ref{lem:HOBF}. Note the logarithmic y-axis. $1000$ bit precision numbers were used for these calculations.}
	\end{subfigure}
	\hfill
	\begin{subfigure}{0.48\textwidth}
		\includegraphics[width=\textwidth]{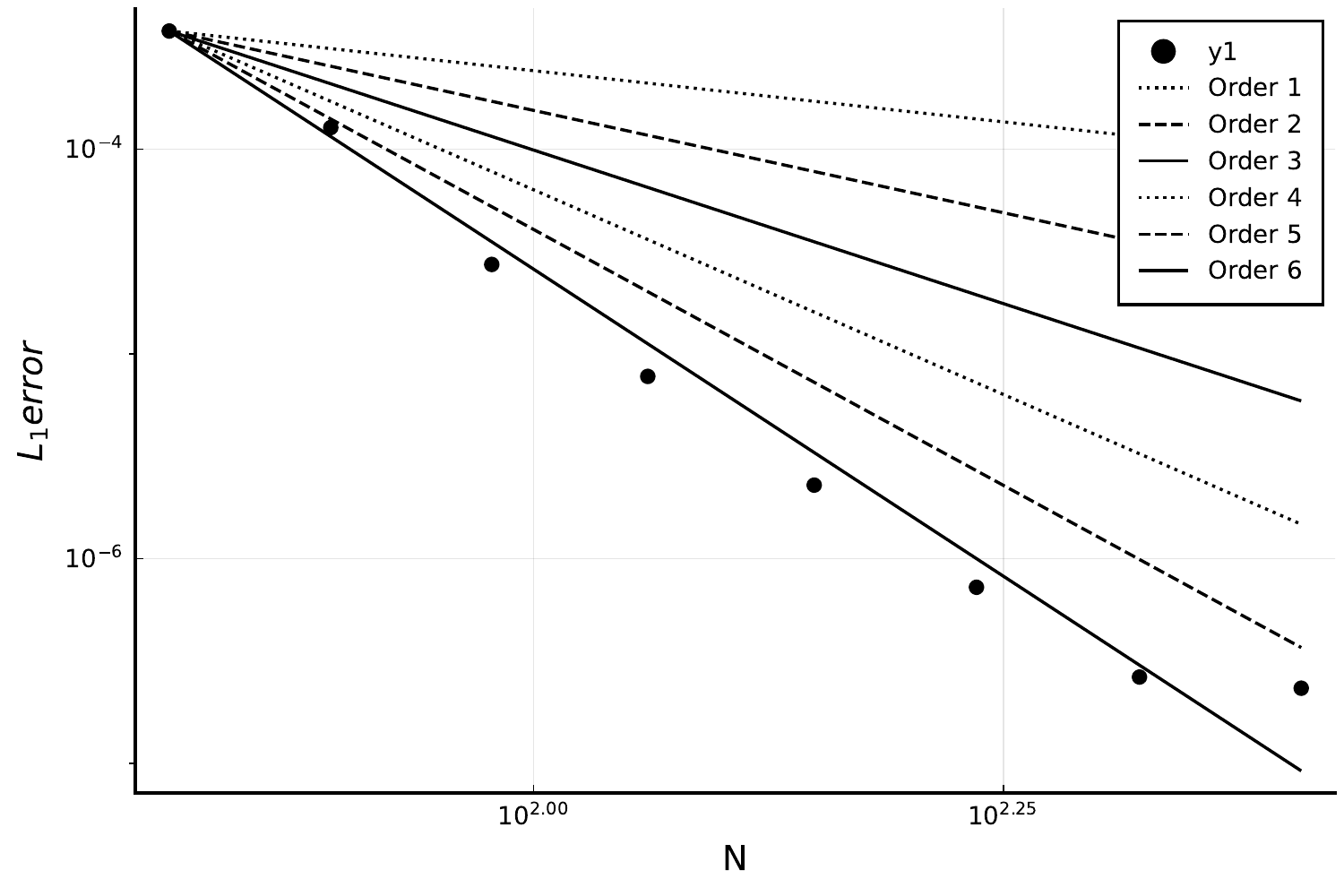}
		\subcaption{Convergence analysis for $p=4$ using only sixth order accurate boundary fluxes. The higher order of accuracy predicted is not demonstrated. Still, the error at the same amount of points is lower than for $p=3$. }
	\end{subfigure}
	\caption{Results concerning the highest order versions and the possible cause in the condition numbers of the used matrices.}
\end{figure}
\subsection{The forward facing step problem}
To underline the abilites of the method a more challenging problem was selected to demonstrate the usage of inflow, reflective and outflow boundary conditions for the Euler equations
	\begin{equation}
	\begin{aligned}
	\derive{u}{t} + \derive{f}{x} + \derive{g}{y} &= 0 \\
u =\begin{pmatrix} \rho \\ \rho v_x \\ \rho v_y\\ E\end{pmatrix},
\quad f\begin{pmatrix} \rho \\ \rho v_x\\ \rho v_y\\ E\end{pmatrix} &= \begin{pmatrix} \rho v_x \\  \rho v_x^2 + p\\ \rho v_x v_y \\ v_x(E + p) \end{pmatrix} ,
\quad g\begin{pmatrix} \rho \\ \rho v_x  \\ \rho v_y\\ E\end{pmatrix} = \begin{pmatrix} \rho v_y \\  \rho v_x v_y \\ \rho v_y^2 + p\\ v_y(E + p) \end{pmatrix} 
	\end{aligned}
\end{equation}
 in two dimensions. The entropy conservative fluxes from \cite{RCFM2016ES} were used in conjunction with the entropy inequality predictors from \cite{klein2022using} to steer the dissipation. As a testcase, the forward facing step problem was selected \cite{WC1984Numerical}. The domain of this problem is built out of a rectangular base domain of length $3$ and height $1$, from which a rectangular step was cutted out, as depicted in figure \ref{fig:ffsdomain}. The step has height $\frac 1 5$ and starts $\frac 3 5$ units into the domain. At the beginning of the calculation, the domain is filled with gas with the state
 \[
 	\rho(x, y, 0) = 1.4, \quad v_x(x, y, 0) = 3.0, \quad v_y(x, y, 0) = 0.0, \quad p(x, y, 0) = 1.0.
 \]
 Gas with this state is fed into the domain from the left, while the top and bottom of the domain use reflective boundary conditions. The right end of the domain uses outflow boundary conditions. The most problematic part is the vertical section of the step as at this point reflective boundary conditions reflect the incoming flow in $x$ direction next to the edge that reflect into the $y$ direction. While in \cite{WC1984Numerical} special cosmetic repairs were applied in this area, we did not use any special algorithm for the corner of the edge. Instead, the boundary conditions as explained in subsection \ref{ssec:bcs} were used at the respective corners. The simulation was carried out with a CFL number of $0.3$ and the SSPRK33 time integration algorithm up to $t = 4.0$. The result can be seen in figure \ref{fig:ffstest}. A grid of $80$ by $240$ cells was used.
 
 \begin{figure}[h]
 	\centering
 	\begin{tikzpicture}[scale = 3.0]
 		\draw (0.0, 0.0) -- (0.6, 0.0) -- (0.6, 0.2) -- (3.0, 0.2) -- (3.0, 1.0) -- (0.0, 1.0) -- (0.0, 0.0);
 		\draw [->] (0.3, 0.5) -- (1.0, 0.5) node [below] {$v_\mathrm{in}$};
 	\end{tikzpicture}
 	\caption{Setup for the forward facing step problem.}
 	\label{fig:ffsdomain}
 	\end{figure}
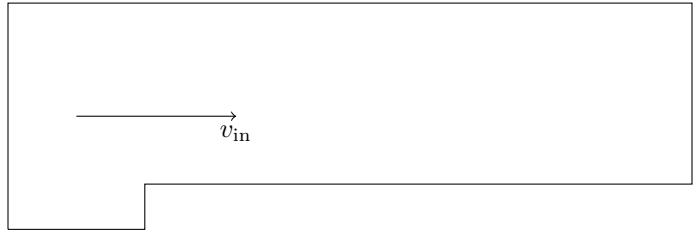
 
 \begin{figure}
 	\begin{subfigure}{0.98\textwidth}
 		\includegraphics[width=0.98 \textwidth]{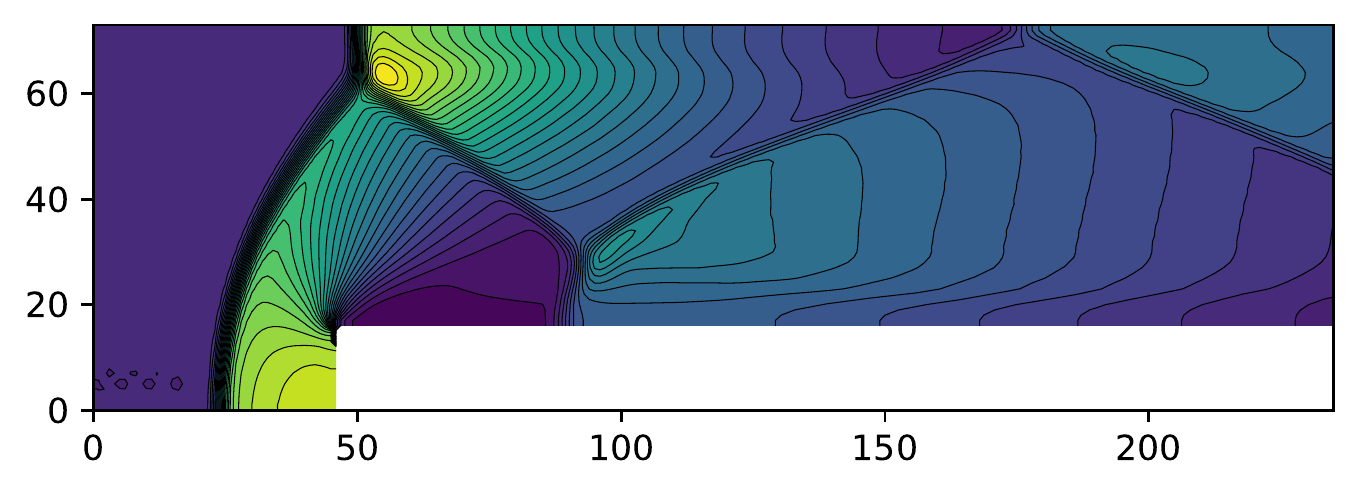}
 		\caption{Density at $t=3.0$, 30 equidistributed contour lines}
 	\end{subfigure}
 \begin{subfigure}{0.98\textwidth}
 	\includegraphics[width=0.98 \textwidth]{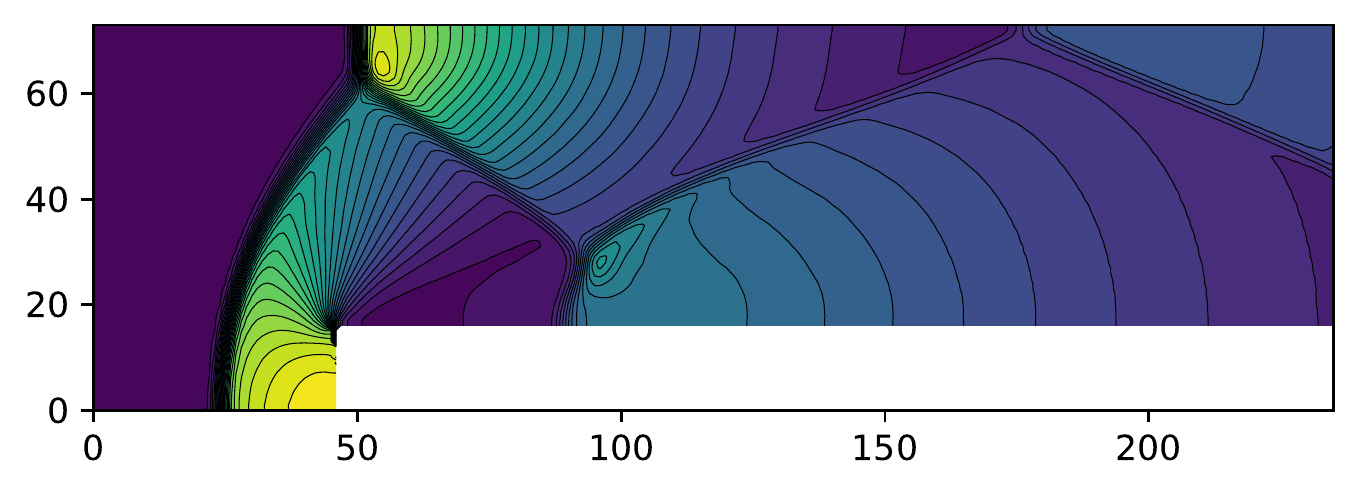}
 	\caption{Pressure at $t=3.0$, 30 equidistributed contour lines}
 \end{subfigure}
 	\caption{Solution to the forward facing step problem at $t = 3.0$. A grid of $80 \times 240$ cells, including the ghost cells, was used.}
 	\label{fig:ffstest}
 \end{figure}

%% file: 6_conclusion.tex
\section{Conclusion}
In this publication entropy conservative fluxes near boundaries were constructed. Afterwards the Lax-Wendroff theorem was verified for entropy dissipative schemes based on these fluxes. The accuracy of the designed fluxes and their ability to implement boundary conditions were afterwards demonstrated in numerical experiments. The designed fluxes are not only usable for the schemes presented in this publication, but could be also used in TeCNO schemes for instant \cite{fjordholm2012arbi}.

%% file: FluxM.tex
\section*{Appendix}
The needed matrices to calculate the boundary fluxes for $p=2$ and $p=3$ are given here for completeness. Sadly, the matrices for $p \geq 4$ seem to have elements with huge denominators, which is why they were not printed. We denote with $A^{p, q}$ the matrix for fluxes of order $2p$ with $q = 0$ for the inner flux, $q > 0$ the right and $q < 0$ the left boundary fluxes. 
\subsection*{Matrices for $p = 2$}
\begin{equation*}
A^{2, -2} =	\left[
	\begin{array}{ccccc}
		0 & \frac{11}{24} & \frac{-1}{12} & 0 & 0 \\
		\frac{11}{24} & 0 & \frac{5}{12} & \frac{-5}{12} & \frac{1}{8} \\
		\frac{-1}{12} & \frac{5}{12} & 0 & 0 & 0 \\
		0 & \frac{-5}{12} & 0 & 0 & 0 \\
		0 & \frac{1}{8} & 0 & 0 & 0 \\
	\end{array}
	\right], 
A^{2, -1} =	\left[
	\begin{array}{ccccc}
		0 & 0 & \frac{-1}{12} & 0 & 0 \\
		0 & 0 & \frac{2}{3} & \frac{-1}{12} & 0 \\
		\frac{-1}{12} & \frac{2}{3} & 0 & 0 & 0 \\
		0 & \frac{-1}{12} & 0 & 0 & 0 \\
		0 & 0 & 0 & 0 & 0 \\
	\end{array}
	\right]
\end{equation*}
\begin{equation*}
A^{2, 0} = \left[
	\begin{array}{cccc}
		0 & 0 & \frac{-1}{12} & 0 \\
		0 & 0 & \frac{2}{3} & \frac{-1}{12} \\
		\frac{-1}{12} & \frac{2}{3} & 0 & 0 \\
		0 & \frac{-1}{12} & 0 & 0 \\
	\end{array}
	\right], 
A^{2, 1} = \left[
	\begin{array}{ccccc}
		0 & 0 & 0 & 0 & 0 \\
		0 & 0 & 0 & \frac{-1}{12} & 0 \\
		0 & 0 & 0 & \frac{2}{3} & \frac{-1}{12} \\
		0 & \frac{-1}{12} & \frac{2}{3} & 0 & 0 \\
		0 & 0 & \frac{-1}{12} & 0 & 0 \\
	\end{array}
	\right]
\end{equation*}
\begin{equation*}
A^{2, 2} = \left[
	\begin{array}{ccccc}
		0 & 0 & 0 & \frac{1}{8} & 0 \\
		0 & 0 & 0 & \frac{-5}{12} & 0 \\
		0 & 0 & 0 & \frac{5}{12} & \frac{-1}{12} \\
		\frac{1}{8} & \frac{-5}{12} & \frac{5}{12} & 0 & \frac{11}{24} \\
		0 & 0 & \frac{-1}{12} & \frac{11}{24} & 0 \\
	\end{array}
	\right]
\end{equation*}

\subsection*{Matrices for $p = 3$}
\begin{equation*}
A^{3, -3}=\left[
	\begin{array}{ccccccc}
		0 & \frac{137}{360} & \frac{-13}{180} & \frac{1}{60} & 0 & 0 & 0 \\
		\frac{137}{360} & 0 & \frac{161}{120} & \frac{-497}{180} & \frac{287}{120} & \frac{-31}{30} & \frac{13}{72} \\
		\frac{-13}{180} & \frac{161}{120} & 0 & \frac{7}{36} & \frac{-7}{30} & \frac{7}{60} & \frac{-1}{45} \\
		\frac{1}{60} & \frac{-497}{180} & \frac{7}{36} & 0 & 0 & 0 & 0 \\
		0 & \frac{287}{120} & \frac{-7}{30} & 0 & 0 & 0 & 0 \\
		0 & \frac{-31}{30} & \frac{7}{60} & 0 & 0 & 0 & 0 \\
		0 & \frac{13}{72} & \frac{-1}{45} & 0 & 0 & 0 & 0 \\
	\end{array}
	\right],
A^{3, -2} =	\left[
	\begin{array}{ccccccc}
		0 & 0 & \frac{-13}{180} & \frac{1}{60} & 0 & 0 & 0 \\
		0 & 0 & \frac{19}{30} & \frac{-3}{20} & \frac{1}{60} & 0 & 0 \\
		\frac{-13}{180} & \frac{19}{30} & 0 & \frac{7}{36} & \frac{-7}{30} & \frac{7}{60} & \frac{-1}{45} \\
		\frac{1}{60} & \frac{-3}{20} & \frac{7}{36} & 0 & 0 & 0 & 0 \\
		0 & \frac{1}{60} & \frac{-7}{30} & 0 & 0 & 0 & 0 \\
		0 & 0 & \frac{7}{60} & 0 & 0 & 0 & 0 \\
		0 & 0 & \frac{-1}{45} & 0 & 0 & 0 & 0 \\
	\end{array}
	\right]
\end{equation*}
\begin{equation*}
A^{3, -1} =	\left[
	\begin{array}{ccccccc}
		0 & 0 & 0 & \frac{1}{60} & 0 & 0 & 0 \\
		0 & 0 & 0 & \frac{-3}{20} & \frac{1}{60} & 0 & 0 \\
		0 & 0 & 0 & \frac{3}{4} & \frac{-3}{20} & \frac{1}{60} & 0 \\
		\frac{1}{60} & \frac{-3}{20} & \frac{3}{4} & 0 & 0 & 0 & 0 \\
		0 & \frac{1}{60} & \frac{-3}{20} & 0 & 0 & 0 & 0 \\
		0 & 0 & \frac{1}{60} & 0 & 0 & 0 & 0 \\
		0 & 0 & 0 & 0 & 0 & 0 & 0 \\
	\end{array}
	\right], 
A^{3, 0} =	\left[
	\begin{array}{cccccc}
		0 & 0 & 0 & \frac{1}{60} & 0 & 0 \\
		0 & 0 & 0 & \frac{-3}{20} & \frac{1}{60} & 0 \\
		0 & 0 & 0 & \frac{3}{4} & \frac{-3}{20} & \frac{1}{60} \\
		\frac{1}{60} & \frac{-3}{20} & \frac{3}{4} & 0 & 0 & 0 \\
		0 & \frac{1}{60} & \frac{-3}{20} & 0 & 0 & 0 \\
		0 & 0 & \frac{1}{60} & 0 & 0 & 0 \\
	\end{array}
	\right]
\end{equation*}
\begin{equation*}
A^{3, 1} =	\left[
	\begin{array}{ccccccc}
		0 & 0 & 0 & 0 & 0 & 0 & 0 \\
		0 & 0 & 0 & 0 & \frac{1}{60} & 0 & 0 \\
		0 & 0 & 0 & 0 & \frac{-3}{20} & \frac{1}{60} & 0 \\
		0 & 0 & 0 & 0 & \frac{3}{4} & \frac{-3}{20} & \frac{1}{60} \\
		0 & \frac{1}{60} & \frac{-3}{20} & \frac{3}{4} & 0 & 0 & 0 \\
		0 & 0 & \frac{1}{60} & \frac{-3}{20} & 0 & 0 & 0 \\
		0 & 0 & 0 & \frac{1}{60} & 0 & 0 & 0 \\
	\end{array}
	\right],
A^{3, 2} =	\left[
	\begin{array}{ccccccc}
		0 & 0 & 0 & 0 & \frac{-1}{45} & 0 & 0 \\
		0 & 0 & 0 & 0 & \frac{7}{60} & 0 & 0 \\
		0 & 0 & 0 & 0 & \frac{-7}{30} & \frac{1}{60} & 0 \\
		0 & 0 & 0 & 0 & \frac{7}{36} & \frac{-3}{20} & \frac{1}{60} \\
		\frac{-1}{45} & \frac{7}{60} & \frac{-7}{30} & \frac{7}{36} & 0 & \frac{19}{30} & \frac{-13}{180} \\
		0 & 0 & \frac{1}{60} & \frac{-3}{20} & \frac{19}{30} & 0 & 0 \\
		0 & 0 & 0 & \frac{1}{60} & \frac{-13}{180} & 0 & 0 \\
	\end{array}
	\right]
\end{equation*}
\begin{equation*}
A^{3, 3} =	\left[
	\begin{array}{ccccccc}
		0 & 0 & 0 & 0 & \frac{-1}{45} & \frac{13}{72} & 0 \\
		0 & 0 & 0 & 0 & \frac{7}{60} & \frac{-31}{30} & 0 \\
		0 & 0 & 0 & 0 & \frac{-7}{30} & \frac{287}{120} & 0 \\
		0 & 0 & 0 & 0 & \frac{7}{36} & \frac{-497}{180} & \frac{1}{60} \\
		\frac{-1}{45} & \frac{7}{60} & \frac{-7}{30} & \frac{7}{36} & 0 & \frac{161}{120} & \frac{-13}{180} \\
		\frac{13}{72} & \frac{-31}{30} & \frac{287}{120} & \frac{-497}{180} & \frac{161}{120} & 0 & \frac{137}{360} \\
		0 & 0 & 0 & \frac{1}{60} & \frac{-13}{180} & \frac{137}{360} & 0 \\
	\end{array}
	\right]
\end{equation*}

%% file: main_gtlw.bbl
\begin{thebibliography}{10}

\bibitem{abgrall2018general}
R.~Abgrall.
\newblock A general framework to construct schemes satisfying additional
  conservation relations. {Application} to entropy conservative and entropy
  dissipative schemes.
\newblock {\em J. Comput. Phys.}, 372:640--666, 2018.

\bibitem{abgrall2021relaxation}
R{\'e}mi Abgrall, Elise~Le M{\'e}l{\'e}do, Philipp {\"O}ffner, and Davide
  Torlo.
\newblock Relaxation deferred correction methods and their applications to
  residual distribution schemes.
\newblock {\em accepted in Smai-JCM}, 2022.

\bibitem{abgrall2021analysis_II}
R{\'e}mi Abgrall, Jan Nordstr{\"o}m, Philipp {\"O}ffner, and Svetlana Tokareva.
\newblock Analysis of the {SBP-SAT} stabilization for finite element methods
  part {II}: Entropy stability.
\newblock {\em Commun. Appl. Math. Comput.}, pages 1--23, 2021.

\bibitem{abgrall2022reinterpretation}
R{\'e}mi Abgrall, Philipp {\"O}ffner, and Hendrik Ranocha.
\newblock Reinterpretation and extension of entropy correction terms for
  residual distribution and discontinuous {G}alerkin schemes: Application to
  structure preserving discretization.
\newblock {\em J. Comput. Phys.}, 453:110955, 2022.

\bibitem{abgrall2016handbook}
Remi Abgrall and Chi-Wang Shu, editors.
\newblock {\em Handbook of numerical methods for hyperbolic problems. {Basic}
  and fundamental issues}, volume~17 of {\em Handb. Numer. Anal.}
\newblock Amsterdam: Elsevier/North Holland, 2016.

\bibitem{carpenter2016entropy}
MH~Carpenter, TC~Fisher, EJ~Nielsen, Matteo Parsani, M~Sv{\"a}rd, and
  N~Yamaleev.
\newblock Entropy stable summation-by-parts formulations for compressible
  computational fluid dynamics.
\newblock In {\em Handbook of Numerical Analysis}, volume~17, pages 495--524.
  Elsevier, 2016.

\bibitem{chan2018discretely}
Jesse Chan.
\newblock On discretely entropy conservative and entropy stable discontinuous
  {G}alerkin methods.
\newblock {\em Journal of Computational Physics}, 362:346--374, 2018.

\bibitem{chen2017entropy}
Tianheng Chen and Chi-Wang Shu.
\newblock Entropy stable high order discontinuous {G}alerkin methods with
  suitable quadrature rules for hyperbolic conservation laws.
\newblock {\em Journal of Computational Physics}, 345:427--461, 2017.

\bibitem{clain2011high}
S.~{Clain}, S.~{Diot}, and R.~{Loub\`ere}.
\newblock A high-order finite volume method for systems of conservation
  laws-multi-dimensional optimal order detection {(MOOD)}.
\newblock {\em {J. Comput. Phys.}}, 230(10):4028--4050, 2011.

\bibitem{crean2018entropy}
Jared Crean, Jason~E Hicken, David C Del~Rey Fern{\'a}ndez, David~W Zingg, and
  Mark~H Carpenter.
\newblock Entropy-stable summation-by-parts discretization of the {E}uler
  equations on general curved elements.
\newblock {\em Journal of Computational Physics}, 356:410--438, 2018.

\bibitem{fisher2013high_2}
Travis~C. Fisher and Mark~H. Carpenter.
\newblock High-order entropy stable finite difference schemes for nonlinear
  conservation laws: finite domains.
\newblock {\em J. Comput. Phys.}, 252:518--557, 2013.

\bibitem{fisher2013discretely}
Travis~C. Fisher, Mark~H. Carpenter, Jan Nordstr{\"o}m, Nail~K. Yamaleev, and
  Charles Swanson.
\newblock Discretely conservative finite-difference formulations for nonlinear
  conservation laws in split form: theory and boundary conditions.
\newblock {\em J. Comput. Phys.}, 234:353--375, 2013.

\bibitem{fjordholm2012arbi}
Ulrik~S. Fjordholm, Siddhartha Mishra, and Eitan Tadmor.
\newblock Arbitrarily high-order accurate entropy stable essentially
  nonoscillatory schemes for systems of conservation laws.
\newblock {\em SIAM J. Numer. Anal.}, 50(2):544--573, 2012.

\bibitem{fjordholm2012entropy}
Ulrik~S. Fjordholm, Siddhartha Mishra, and Eitan Tadmor.
\newblock Entropy stable {ENO} scheme.
\newblock In {\em Hyperbolic problems. Theory, numerics and applications. Vol.
  1. Proceedings of the 13th international conference on hyperbolic problems,
  HYP 2010, Beijing, China, June 15--19, 2010}, pages 12--27. Hackensack, NJ:
  World Scientific; Beijing: Higher Education Press, 2012.

\bibitem{gassner2016split}
Gregor~J Gassner, Andrew~R Winters, and David~A Kopriva.
\newblock Split form nodal discontinuous {G}alerkin schemes with
  summation-by-parts property for the compressible {E}uler equations.
\newblock {\em Journal of Computational Physics}, 327:39--66, 2016.

\bibitem{Gustafsson1975Convergence}
Bertil Gustafsson.
\newblock The convergence rate for difference approximations to mixed initial
  boundary value problems.
\newblock {\em Mathematics of Computation}, 29(130):396--406, 1975.

\bibitem{Gustafsson1981Convergence}
Bertil Gustafsson.
\newblock The convergence rate for difference approximations to general mixed
  initial boundary value problems.
\newblock {\em SIAM Journal on Numerical Analysis}, 18(2):179--190, 1981.

\bibitem{Harten83b}
Amiram Harten.
\newblock On the symmetric form of systems of conservation laws with entropy.
\newblock {\em Journal of Computational Physics}, 49:151--164, 1983.

\bibitem{ENOIII}
Amiram Harten, Björn Enquist, Stanley Osher, and Sukumar~R. Chakravarthy.
\newblock Uniformly high order accurate essentially non-oscillatory schemes
  {III}.
\newblock {\em Journal of Computational Physics}, 71:231--303, 1987.

\bibitem{hesthaven2019entropy}
Jan~S Hesthaven and Fabian M{\"o}nkeberg.
\newblock Entropy stable essentially nonoscillatory methods based on {RBF}
  reconstruction.
\newblock {\em ESAIM: Mathematical Modelling and Numerical Analysis},
  53(3):925--958, 2019.

\bibitem{jameson2008construction}
Antony Jameson.
\newblock The construction of discretely conservative finite volume schemes
  that also globally conserve energy or entropy.
\newblock {\em J. Sci. Comput.}, 34(2):152--187, 2008.

\bibitem{klein2022stabilizing}
Simon-Christian Klein.
\newblock Stabilizing discontinuous galerkin methods using dafermos' entropy
  rate criterion.
\newblock {\em arXiv preprint arXiv:2208.00941}, 2022.

\bibitem{klein2022using}
Simon-Christian Klein.
\newblock Using the dafermos entropy rate criterion in numerical schemes.
\newblock {\em BIT Numerical Mathematics}, pages 1--29, 2022.

\bibitem{kuzmin2022bound}
Dmitri Kuzmin, Manuel Quezada~de Luna, David~I. Ketcheson, and Johanna
  Gr{\"u}ll.
\newblock Bound-preserving flux limiting for high-order explicit
  {Runge}-{Kutta} time discretizations of hyperbolic conservation laws.
\newblock {\em J. Sci. Comput.}, 91(1):34, 2022.
\newblock Id/No 21.

\bibitem{Lax71}
Peter~D. Lax.
\newblock Shock waves and entropy.
\newblock {\em Contributions to Nonlinear Functional Analysis}, pages 603--634,
  1971.

\bibitem{LW1960}
Peter~D. Lax and B.~Wendroff.
\newblock Systems of conservation laws.
\newblock {\em Communications on Pure and Applied Mathematics}, 13:217--237,
  1960.

\bibitem{lefloch2002fully}
Philippe~G LeFloch, Jean-Marc Mercier, and Christian Rohde.
\newblock Fully discrete, entropy conservative schemes of arbitrary order.
\newblock {\em SIAM Journal on Numerical Analysis}, 40(5):1968--1992, 2002.

\bibitem{ranocha2018comparison}
Hendrik Ranocha.
\newblock Comparison of some entropy conservative numerical fluxes for the
  {E}uler equations.
\newblock {\em Journal of Scientific Computing}, 76(1):216--242, 2018.

\bibitem{ranocha2018generalised}
Hendrik Ranocha.
\newblock {\em Generalised summation-by-parts operators and entropy stability
  of numerical methods for hyperbolic balance laws}.
\newblock G{\"o}ttingen: Cuvillier; Braunschweig: TU Braunschweig (Diss.),
  2018.

\bibitem{ranocha2020relaxation}
Hendrik Ranocha, Mohammed Sayyari, Lisandro Dalcin, Matteo Parsani, and David~I
  Ketcheson.
\newblock Relaxation {R}unge--{K}utta methods: Fully discrete explicit
  entropy-stable schemes for the compressible {E}uler and {N}avier--{S}tokes
  equations.
\newblock {\em SIAM Journal on Scientific Computing}, 42(2):A612--A638, 2020.

\bibitem{ray2016entropy}
Deep Ray, Praveen Chandrashekar, Ulrik~S. Fjordholm, and Siddhartha Mishra.
\newblock Entropy stable scheme on two-dimensional unstructured grids for
  {Euler} equations.
\newblock {\em Commun. Comput. Phys.}, 19(5):1111--1140, 2016.

\bibitem{RCFM2016ES}
Deep Ray, Praveen Chandrashekar, Ulrik~S. Fjordholm, and Siddhartha Mishra.
\newblock Entropy stable scheme on two-dimensional unstructured grids for euler
  equations.
\newblock {\em Communications in Computational Physics}, 19(5):1111–1140,
  2016.

\bibitem{shi2018local}
Cengke Shi and Chi-Wang Shu.
\newblock On local conservation of numerical methods for conservation laws.
\newblock {\em Comput. Fluids}, 169:3--9, 2018.

\bibitem{SO1988}
Chi-Wang Shu and Stanley Osher.
\newblock Efficient implementation of essentially nonoscillatory
  shock-capturing schemes.
\newblock {\em J. Comput. Phys.}, 77(2):439--471, 1988.

\bibitem{SO1989}
Chi-Wang Shu and Stanley Osher.
\newblock Efficient implementation of essentially nonoscillatory
  shock-capturing schemes. {II}.
\newblock {\em J. Comput. Phys.}, 83(1):32--78, 1989.

\bibitem{sonar1992entropy}
T.~Sonar.
\newblock Entropy production in second-order three-point schemes.
\newblock {\em Numer. Math.}, 62(3):371--390, 1992.

\bibitem{SonarENO}
Thomas Sonar.
\newblock {\em Mehrdimensionale {ENO} Verfahren}.
\newblock Teubner, 1997.

\bibitem{tadmor1987numerical}
Eitan Tadmor.
\newblock The numerical viscosity of entropy stable schemes for systems of
  conservation laws. {I}.
\newblock {\em Mathematics of Computation}, 49(179):91--103, 1987.

\bibitem{tadmor2003entropy}
Eitan Tadmor.
\newblock Entropy stability theory for difference approximations of nonlinear
  conservation laws and related time-dependent problems.
\newblock {\em Acta Numerica}, 12:451--512, 2003.

\bibitem{Toro1997Riemann}
Eleuterio~F. Toro.
\newblock {\em Riemann solvers and numerical methods for fluid dynamics. {A}
  practical introduction}.
\newblock Berlin: Springer, 1997.

\bibitem{WC1984Numerical}
Paul Woodward and Phillip Colella.
\newblock The numerical simulation of two-dimensional fluid flow with strong
  shocks.
\newblock {\em J. Comput. Phys.}, 54:115--173, 1984.

\end{thebibliography}
